\documentclass[12pt,reqno]{amsart}
\usepackage{amsmath}
\usepackage{amsfonts}
\usepackage{amssymb}
\usepackage[all,cmtip]{xy}           %xypic macro for latex2.09

\usepackage{bbm}
\usepackage{bbding}
\usepackage{txfonts}
\usepackage[shortlabels]{enumitem}
\usepackage{ifpdf}
\ifpdf
\usepackage[colorlinks,final,backref=page,hyperindex]{hyperref}
\else
\usepackage[colorlinks,final,backref=page,hyperindex,hypertex]{hyperref}
\fi
\usepackage{tikz-cd}
\usepackage[active]{srcltx}
\usepackage{tikz}
\usepackage{amscd}
\usepackage{bm}
\usepackage{mathrsfs}

\makeatletter

\newtheorem{Thm}{Theorem}[section]
\newtheorem{Prop}[Thm]{Proposition}
\newtheorem{Lem}[Thm]{Lemma}
\newtheorem{Cor}[Thm]{Corollary}
\newtheorem{prop-def}[Thm]{Proposition-Definition}

\theoremstyle{definition}
\newtheorem{Def}[Thm]{Definition}

\newtheorem{Remark}[Thm]{Remark}
\newtheorem{Exam}[Thm]{Example}
\newtheorem{Ques}[Thm]{Question}

\newtheorem{Crit}[Thm]{Criterion}

\newcommand{\nc}{\newcommand}

\nc{\delete}[1]{{}}

\nc{\mlabel}[1]{\label{#1}}  % Use this to suppress names
\nc{\mcite}[1]{\cite{#1}}  % Use this to suppress names
\nc{\mref}[1]{\ref{#1}}  % Use this to suppress names
\nc{\meqref}[1]{\eqref{#1}}  % Use this to suppress names
\nc{\mbibitem}[1]{\bibitem{#1}} % Use this to show number
\delete{
	\nc{\mlabel}[1]{\label{#1} {{\emph{{\ }\ (#1)}}}}				 % Use this lines to show names
	\nc{\mcite}[1]{\cite{#1}{{\emph{{\ }(#1)}}}}  % Use this lines to show names
	\nc{\mref}[1]{\ref{#1}{{\emph{{\ }(#1)}}}}  % Use this lines to show names
	\nc{\meqref}[1]{\eqref{#1}{{\emph{{\ }(#1)}}}}  % Use this lines to show names
	\nc{\mbibitem}[1]{\bibitem[\bf #1]{#1}} % Use this to show name
}

\nc{\mrm}[1]{{\rm #1}}
\nc{\name}[1]{{\bf #1}}
\nc{\tforall}{\ \text{for all }}

\nc{\la}{\longrightarrow}
\nc{\ot}{\otimes}
\nc{\rar}{\rightarrow}

\nc{\ac}{\mathrm{\textup{!`}}}

\nc{\Mod}{\mathrm{Mod~}}
\nc{\proj}{\mathrm{proj}}
\nc{\inj}{\mathrm{inj}}
\nc{\Proj}{\mathrm{Proj}}
\nc{\Alg}{{\mathrm{Alg}}}
\nc{\bbL}{{\mathbb{L}}}
\nc{\bbR}{{\mathbb{R}}}
\nc{\bfk}{{\bf k}}
\nc{\C}{{\mathrm{C}}}
\nc{\Cone}{\mrm{Cone}}
\nc{\DA}{{\mathsf{DA}_\lambda}}
\nc{\Dif}{{{}_\lambda\!\mathfrak{Dif}}}
\nc{\Difinfty}{{{}_\lambda\!\mathfrak{Dif}_\infty}}
\nc{\DO}{{\mathsf{DO}_\lambda}}
\nc{\dg}{\mrm{dg}}
\nc{\End}{\mrm{End}}
\nc{\Ext}{\mrm{Ext}}
\nc{\Fil}{\mrm{Fil}}
\nc{\Fr}{\mrm{Fr}}
\nc{\Frob}{\mrm{Frob}}
\nc{\Gal}{\mrm{Gal}}
\nc{\GL}{\mrm{GL}}
\nc{\Hom}{\mrm{Hom}}
\nc{\Tor}{\mrm{Tor}}
\nc{\Hoch}{\mrm{Hoch}}
\nc{\HC}{\mrm{HC}}
\nc{\hsr}{\mrm{H}}
\nc{\hpol}{\mrm{HP}}
\nc{\im}{\mrm{Im}}
\nc{\K}{\mrm{K}}
\nc{\Ker}{\mrm{Ker}}
\nc{\rad}{\mrm{rad}}
\nc{\pd}{\mrm{pd}}
\nc{\Id}{\mrm{Id}}

\nc{\Irr}{\mrm{Irr}}
\nc{\incl}{\mrm{incl}}
\nc{\length}{\mrm{length}}
\nc{\NLSW}{\mrm{NLSW}}
\nc{\Lie}{\mrm{Lie}}
\nc{\mchar}{\rm char}
\nc{\mpart}{\mrm{part}}
\nc{\ql}{{\QQ_\ell}}
\nc{\qp}{{\QQ_p}}
\nc{\rank}{\mrm{rank}}
\nc{\thick}{\mrm{thick}}
\nc{\rcot}{\mrm{cot}}
\nc{\rdef}{\mrm{def}}
\nc{\rdiv}{{\rm div}}
\nc{\rmH}{ {\mathrm{H}}}
\nc{\rtf}{{\rm tf}}
\nc{\rtor}{{\rm tor}}
\nc{\res}{\mrm{res}}
\nc{\Sh}{{\mathrm{Sh}}}
\nc{\SL}{\mrm{SL}}
\nc{\Spec}{\mrm{Spec}}
\nc{\sg}{{\mathrm{sg}}}
\nc{\sgn}{{\mathrm{sgn}}}
\nc{\tor}{\mrm{tor}}
\nc{\Tr}{\mrm{Tr}}
\nc{\tr}{\mrm{tr}}
\nc{\wt}{\mrm{wt}}
\nc{\op}{\mrm{op}}

\nc{\cpx}[1]{#1^{\bullet}}

\nc{\MAA}{{\mathbb A}}   \nc{\BB}{{\mathbb B}}
\nc{\CC}{{\mathbb C}}
\nc{\DD}{{\mathbb D}}   \nc{\EE}{{\mathbb E}}
\nc{\FF}{{\mathbb F}}   \nc{\GG}{{\mathbb G}}
\nc{\HH}{ \mathrm{HH}}   \nc{\LL}{{\mathbb L}}
\nc{\NN}{{\mathbb N}}   \nc{\PP}{{\mathbb P}}
\nc{\QQ}{{\mathbb Q}}   \nc{\RR}{{\mathbb R}}
\nc{\TT}{{\mathbb T}}   \nc{\VV}{{\mathbb V}}
\nc{\ZZ}{{\mathbb Z}}   \nc{\TP}{\widetilde{P}}
\nc{\m}{{\mathbbm m}}   \nc{\MSS}{{\mathbb S}}

\nc{\cala}{{\mathcal A}}    \nc{\calc}{{\mathcal C}}
\nc{\cald}{\mathcal{D}}     \nc{\cale}{{\mathcal E}}
\nc{\calf}{{\mathcal F}}    \nc{\calg}{{\mathcal G}}
\nc{\calh}{{\mathcal H}}    \nc{\cali}{{\mathcal I}}
\nc{\call}{{\mathcal L}}    \nc{\calm}{{\mathcal M}}
\nc{\caln}{{\mathcal N}}    \nc{\calo}{{\mathcal O}}
\nc{\calp}{{\mathcal P}}    \nc{\calr}{{\mathcal R}}
\nc{\cals}{{\mathcal S}}    \nc{\calt}{{\mathcal T}}
\nc{\calv}{{\mathcal V}}    \nc{\calw}{{\mathcal W}}
\nc{\calx}{{\mathcal X}}

\nc{\fraka}{{\mathfrak a}}
\nc{\frakb}{\mathfrak{b}}
\nc{\frakg}{{\frak g}}
\nc{\frakl}{{\frak l}}
\nc{\fraks}{{\frak s}}
\nc{\frakB}{{\frak B}}
\nc{\frakm}{{\frak m}}
\nc{\frakM}{{\frak M}}
\nc{\frakp}{{\frak p}}
\nc{\frakW}{{\frak W}}
\nc{\frakX}{{\frak X}}
\nc{\frakS}{{\frak S}}
\nc{\frakA}{{\frak A}}
\nc{\frakC}{{\frak{C}}}
\nc{\frakx}{{\frakx}}
\nc{\frakt}{{\mathfrak{T}}}

\nc{\vspa}{\vspace{-.1cm}}
\nc{\vspb}{\vspace{-.2cm}}
\nc{\vspc}{\vspace{-.3cm}}
\nc{\vspd}{\vspace{-.4cm}}
\nc{\vspe}{\vspace{-.5cm}}

\nc{\ra}{\rightarrow}
\nc{\laa}{\leftarrow}
\nc{\lra}{\longrightarrow}
\nc{\lraf}[1]{\stackrel{#1}{\lra}}
\nc{\raf}[1]{\stackrel{#1}{\ra}}
\nc{\laf}[1]{\stackrel{#1}{\laa}}

\nc{\D}[1]{\mathrm{D}(#1)}
\nc{\Db}[1]{\mathrm{D}^b(#1)}

\nc{\Kb}[1]{{\mathrm K}^b(#1)}

\nc{\per}{\mathsf {per}}
\nc{\modcat}[1]{\ensuremath{\text{{\rm mod}}}(#1)}
\nc{\Modcat}[1]{\ensuremath{\text{{\rm Mod}}}(#1)}
\nc{\pmodcat}[1]{\mbox{{\rm proj}}(#1)}
\nc{\imodcat}[1]{\mbox{{\rm inj}}(#1)}
\nc{\gpmodcat}[1]{\mbox{{\rm Gproj}}(#1)}
\nc{\stmodcat}[1]{\mbox{{\rm {\underline{mod}}}}(#1)}
\nc{\stgp}[1]{\mbox{{\rm {\underline{Gproj}}}}(#1)}

\nc{\Sg}{\mathrm{D}^b_{\mathrm{sg}}}
\nc{\Gdef}{\mathrm{D}^b_{\mathrm{def}}}
\nc{\Res}{\mathrm{Res}}
\nc{\Gproj}{\mathrm{Gproj}}
\nc{\fGd}[1]{\mathrm{D}^b_{\rm fGd}(#1)}

\nc{\li}[1]{\textcolor{red}{#1 }}
\nc{\lir}[1]{\textcolor{red}{\underline{Li:}#1 }}
\nc{\gd}[1]{\textcolor{purple}{#1 }}
\nc{\kai}[1]{\textcolor{blue}{#1}}

\begin{document}
\title[Bounded extensions]{Categorical properties and homological conjectures for bounded extensions of algebras}

\author{Yongyun Qin, Xiaoxiao Xu, Jinbi Zhang* and Guodong Zhou}
\address{Yongyun Qin, School of Mathematics, Yunnan Key Laboratory of Modern Analytical Mathematics and Applications,
Yunnan Normal University, Kunming 650500, China}
\email{qinyongyun2006@126.com}

\address{Xiaoxiao Xu and Guodong Zhou, School of Mathematical Sciences, Key Laboratory of Mathematics and Engineering Applications of Ministry of Education, Shanghai Key Laboratory of PMMP, East China Normal University, Shanghai 200241, China}
\email{52275500015@stu.ecnu.edu.cn, gdzhou@math.ecnu.edu.cn}

\address{Jinbi Zhang, School of Mathematical Sciences, Anhui University, Hefei 230601, China}
\email{zhangjb@ahu.edu.cn}

\date{\today}

\begin{abstract}
Following Cibils, Lanzilotta, Marcos, and Solotar, an extension $B\subset A$ of finite dimensional algebras is bounded if the $B$-$B$-bimodule $A/B$ is $B$-tensor nilpotent, its projective dimension is finite and $\Tor_i^B(A/B, (A/B)^{\otimes_B j})=0$ for all $i, j\geq 1$. We show that for a bounded extension $B\subset A$, the algebras $A$ and $B$ are singularly equivalent of Morita type with level. Additionally, under mild conditions, their stable categories of Gorenstein projective modules and Gorenstein defect categories are equivalent, respectively. Some homological conjectures are also investigated for bounded extensions, including Auslander-Reiten conjecture,  finitistic dimension conjecture,  Fg condition,  Han's conjecture, and  Keller's conjecture. Applications to  trivial extensions and triangular matrix algebras are given.

In course of proof, we give some handy criteria for a functor between module categories to  induce  triangle functors between stable categories of Gorenstein projective modules and Gorenstein defect categories, which generalise some known criteria,  and hence  might be of independent interest.
\end{abstract}

\subjclass[2010]{
16E35   %Derived categories and associative algebras
16E30   %Homological functors on modules (Tor, Ext, etc.) in associative algebras
16E40   %(co)homology of rings and algebras
16E65   %Homological conditions on associative rings (generalizations of regular, Gorenstein, Cohen-Macaulay rings, etc.)
%16G     %Representation theory of associative rings and algebras
16G50   %Cohen-Macaulay modules in associative algebras
16D20   %Bimodules in associative algebras
%18E35Localization of categories, calculus of fractions
}

\keywords{Gorenstein defect category; Homological conjecture; Ring extension;  Singular equivalence of Morita type with level;  Singularity category.\\
\text{      } *Corresponding author}

\maketitle

\vspace{-.7cm}

 \tableofcontents

\vspace{-.7cm}

\allowdisplaybreaks

\section*{Introduction}
Let $A$ be a finite dimensional algebra over a field $\bfk$. Denote by  $\Modcat{A}$ (resp.  $\modcat{A}$, $\proj (A)$,  $\inj (A)$) the category of left $A$-modules (resp.  finitely generated left $A$-modules, finitely generated projective left $A$-modules, finitely generated   injective  left $A$-modules).
Let $\K(A)=\K(\Modcat{A})$ be the homotopy category of  $\Modcat{A}$, $\Kb{\proj{A}}$ be the bounded  homotopy category of $\proj (A)$,  $\D{A}=\D{\Modcat{A}}$ be  the unbounded derived category of  $\Modcat{A}$, and by $\Db{\modcat{A}}$ be the derived category of bounded complexes over $\modcat{A}$.
The \emph{singularity category} $\Sg(A)$ of $A$ is defined to be the Verdier quotient
$$\Sg(A):=\Db{\modcat{A}}/\Kb{\proj{A}},$$
see \cite{Buch21,O04}. The singularity category measures the homological singularity of an algebra in the sense that an algebra $A$ has finite global dimension if and only if its singularity category $\Sg(A)$ is trivial.

The singularity category also  captures the stable homological features of an algebra $A$ (\cite{Buch21}). Denote by $\stgp{A}$ the stable category of finitely generated  Gorenstein projective $A$-modules. Following \cite{Buch21}, there is a natural triangle functor $F:\stgp{A}\rightarrow \Sg(A)$  which is always fully faithful, and $F$ is an equivalence if and only if $A$ is Gorenstein (see also \cite{Hap91}). Inspired by this, Bergh, J{\o}rgensen, and Oppermann  (\cite{BJO15}) defined the \emph{Gorenstein defect category}  as the Verdier quotient
$$\Gdef(A):=\Sg(A)/{\rm Im}(F)$$  and proved that $A$ is Gorenstein if and only if $\Gdef(A)=0$. This means that the Gorenstein defect category of $A$ measures how far the algebra $A$ is from being Gorenstein.

The essential image of the functor $F$ is also clear. Denote by $\fGd{A}$ the full subcategory of
$\Db{\modcat{A}}$ formed by all complexes which are  quasi-isomorphic to bounded complexes of Gorenstein projective modules, that is,
   $\fGd{A}$ is   the thick subcategory of $\Db{\modcat{A}}$ generated by $\gpmodcat{A}$,  see \cite[Theorem 2.7]{LHZ22}.
The following triangle equivalence
$$ \stgp{A} \simeq \fGd{A}/\Kb{\pmodcat{A}}$$
is well known, see, for instance,  \cite[Theorem 4.4.1]{Buch21},
 \cite[Theorem 4]{PZ15},  \cite[Lemma 4.1]{CR20}, or in a much more general setup \cite[Theorem 3.7]{YZ22}.
 We will see that the  triangulated category  $\fGd{A}$ plays an important role in this paper.

Two algebras $A$ and $B$ are called \emph{singularly equivalent} provided that there is a triangle equivalence between $\Sg(A)$ and $\Sg(B)$, which  is called a \emph{singular equivalence} between $A$ and $B$. In \cite{Wang15}, Wang introduced the notion of singular equivalences of Morita type with level, which is a generalisation of stable equivalences of Morita type in the sense of Brou\'e \cite{Bro90} and  singular equivalences of Morita type in the sense of Chen and Sun  \cite{CS12} (see also \cite{ZZ13}).
\begin{Def}[\cite{Wang15}]Given two finite dimensional $\bfk$-algebras $A$ and $B$, and let $M$ be an $A$-$B$-bimodule and $N$ be a $B$-$A$-bimodule. We say that $(_AM_B, {_BN_A})$ defines a singular equivalence of Morita type with level $l$ for some $l\in \mathbb{N}$ if $_AM_B$ and $ {_BN_A}$ are projective on both sides, and the following conditions are satisfied:
$$M\otimes_B N\cong \Omega^l_{A^e}(A)\mbox{ in } \stmodcat{A^e}\mbox{ and } N\otimes_AM\cong \Omega^l_{B^e}(B)\mbox{ in }\stmodcat{B^e},$$
where $\Omega^l_{A^e}(-)$ and $\Omega^l_{B^e}(-)$ denote the syzygy endofunctor of the stable module category of $A^e$ and $B^e$, respectively.
 \end{Def}
 In this situation, there is a triangle equivalence $$M\otimes_B-:\Sg(B)\ra \Sg(A)$$ with quasi-inverse
$$[l]\circ (N\otimes_A-):\Sg(A)\ra \Sg(B).$$

This kind of singular equivalences provide rich structural information, and play an important role in the study of singular equivalences and homological properties. For example, it was verified in \cite{Wang15}  and \cite{CLW20} that the finitistic dimension conjecture and Keller's conjecture for singular Hochschild cohomology are invariant under singular equivalences of Morita type with level. Thus we are interested in constructing singular equivalences of Morita type with level.

In \cite{CLW23}, Chen, Liu, and Wang proved that a tensor functor with a bimodule defines a singular equivalence of Morita type with level under certain conditions, and in \cite{Dal21}, Dalezios proved that for certain Gorenstein algebras, a singular equivalence induced from tensoring with a complex of bimodules always induces a singular equivalence of Morita type with level. Qin in \cite{Qin22} gave a complex version of Chen-Liu-Wang's work, which extends Dalezios' result to arbitrary algebras (not limited to Gorenstein algebras).

Given an algebra $A$, there is an exact sequence of triangulated categories:
$$0\ra \stgp{A}\ra \Sg(A)\ra \Gdef(A) \ra 0.$$
A natural and fundamental question is: When do two related algebras share the same stable category of Gorenstein projective modules, singularity category and Gorenstein defect category up to triangle equivalence?
For partial answers to this question, we refer to \cite{Chen09, Chen14,  LHZ22, Lu17, Qin23, ZLM24}. In this paper, we also give a partial answer,
that is, we will compare  stable category of Gorenstein projective modules, singularity category and Gorenstein defect category
between two algebras linked by an algebra extension. From \cite{CLMS22}, an extension $B \subset A$ is called
{\it left (resp. right) bounded} if $A/B$ has finite projective
dimension as a $B$-$B$-bimodule, $A/B$ is projective as a left (resp. right) $B$-module and
some tensor power $(A/B)^{\otimes _Bp}$
is 0. Inspired by this, we consider a more general notion  suggested by Cibils, Lanzilotta, Marcos, and Solotar in \cite{CLMS22erratum}:
\begin{Def}[{\cite{CLMS22erratum}}]\label{def-bound-exte} An extension $B\subset A$ of finite dimensional algebras is called bounded if
\begin{itemize}
\item[$(1)$] there exists $p\geq 1$ such that the tensor power $(A/B)^{\otimes _Bp}$ vanishes;

\item[$(2)$] $A/B$ has finite projective dimension as a $B$-$B$-bimodule;

\item[$(3)$] $\Tor_i^B(A/B, (A/B)^{\otimes_B j})=0$ for all $i, j\geq 1$.

\end{itemize}
\end{Def}
Examples of bounded extension in the sense of the above definition but not in the sense of \cite{CLMS22} are given in
Proposition~\ref{Prop-bounded-exten} and Example~\ref{Ex-bounded-extensions}.

   Let ${\rm Gproj}(A)^{\bot}:=\{X\in \modcat{A}\mid \Ext_{A}^{i}(U,X)=0 \mbox{ for all } U\in \gpmodcat{A},\, i\ge 1\}$.
The first main  result we obtain in this paper  is as follows.
\begin{Thm}[Theorems \ref{main-thm-semt}, \ref{main-thm-semt2},
  \ref{main-thm-def1} and \ref{main-thm-def2}]
\label{main-thm}
Let $B\subset A$ be a bounded  extension. Then
$${_A}A\otimes^{\mathbb{L}}_{B}-: \Sg(B)\rightleftharpoons\Sg(A):  \Res^A_B=_BA\otimes _A- $$ is an adjoint triangle equivalence which is in fact a singular equivalence of Morita type with level between $A$ and $B$, and the functor $\Res^A_B:\modcat{A}\rightarrow \modcat{B}$ is an eventually homological isomorphism in the sense of \cite{PSS14}.

 Moreover, there are triangle equivalences
$${}_AA\otimes_B-: \stgp{B}\ra  \stgp{A}\mbox{ and }  _AA\otimes^{\mathbb{L}}_B-: \Gdef(B)\rightleftharpoons  \Gdef(A): \Res^A_B$$
if  one of the following two conditions holds:
\begin{itemize}
\item[$(1)$] $\mathbb{R}\Hom_B(A,B)$ is quasi-isomorphic to a bounded complex with each term in ${\rm Gproj}(A)^{\bot}$;

\item[$(2)$] $_BB$ is a direct summand of $_BA$. \end{itemize}
\end{Thm}

In course of proof of the above result, we obtain some handy criteria for a functor between module categories to induce  triangle functors between the stable categories of Gorenstein projective  modules, which is our second main result.

\begin{Crit}[{Criteria~\ref{Crit: restr-Gproj2} and \ref{Crit: restr-Gproj3}}]\label{main criteria}
Let $X$ be an $\Gamma$-$\Lambda$-bimodule over two   finite dimensional $\bfk$-algebras $\Lambda$ and $\Gamma$ such that $X$ has finite projective dimension both as left $\Gamma$-module and as right $\Lambda$-module.
Assume that   $\mathbb{R}\Hom_\Gamma(X,\Gamma)$ is quasi-isomorphic to a bounded complex with each term in ${\rm Gproj}(\Lambda)^{\bot}$,  or that there is an integer $t$ such that $\Omega ^t(X\otimes_\Lambda M)\in {}^{\perp} {\Gamma}$ for any $M\in {}^{\perp}\Lambda$. Then
$X\otimes^{\mathbb{L}}_\Lambda-$ induces a triangle functor from $\fGd{\Lambda}$ to $\fGd{{\Gamma}}$, and hence a triangle functor from $\stgp{\Lambda} \simeq \fGd{\Lambda}/\Kb{\pmodcat{\Lambda}}$ to $\stgp{\Gamma} \simeq \fGd{\Gamma}/\Kb{\pmodcat{\Gamma}}$.
\end{Crit}

This result is of independent interest as it generalises some known criteria appeared in the literatures; see \cite{HP17} and \cite[Theorem I (ii)]{OPS19}.

We first apply Theorem \ref{main-thm} to several homological conjectures, which are still open now. These conjectures play important roles in representation theory of finite dimensional algebras and homological algebra.
Our third main result concerns reductions of these homological conjectures via bounded extensions.
Note that for  Han's conjecture  the following result  was firstly proved by  Cibils, Lanzilotta, Marcos, and Solotar in \cite{CLMS22} for left or right bounded extensions.
\begin{Thm}[Theorem~\ref{main-cor-semt-1} and \ref{main-cor-semt-2}]
\label{main-cor}
Let $B\subset A$ be a bounded extension. Then the following statements hold:
\begin{itemize}
\item[$(1)$] $B$ satisfies the finitistic dimension conjecture if and only if so does $A$;

\item[$(2)$] $B$ satisfies Han's conjecture if and only if so does $A$;

\item[$(3)$]  $B$ satisfies Keller's conjecture if and only if so does $A$.
\end{itemize}

\noindent If, in addition, $\mathbb{R}\Hom_{B}(A,B)$ is perfect as left $A$-module
or $_BB$ is a direct summand of $_BA$, then
\begin{itemize}
\item[$(4)$] $B$ satisfies Auslander-Reiten's conjecture if and only if so does $A$;

\item[$(5)$] $B$ satisfies the Fg condition for   Hochschild cohomology if and only if so does $A$. %In this case, they have the same support variety theory.
\end{itemize}
\end{Thm}

%Note that  bounded extensions in our definition is less restrictive than the left or right bounded extensions introduced in \cite{CLMS22} (but the present definition was also suggested by them \cite{CLMS22erratum}, see Remark~\ref{rem:equivalence-tor} (2)). Consequently, Theorem \ref{main-cor} (2) extends the main result  of  \cite{CLMS22}.

Moreover, we apply Theorem \ref{main-thm} to trivial extensions and triangular matrix algebras, and  get several reduction methods on
singularity categories and Gorenstein defect categories;
see Propositions~\ref{prop-trivial-extension} and  \ref{prop-tri-matrix-alg}, and    Examples~\ref{exam-1} and  \ref{exam-2}.
 In particular, since the arrow removal
operation yields a bounded split extension, we reobtain \cite[Theorem A (ii)]{GPS21} and the main result of \cite{EPS22}
in Example~\ref{exam-1}.

The layout of this paper is as follows: Section~\ref{Sect: notion} introduces the notion of   bounded extensions and provides some examples.  Section~\ref{semt} studies singular  equivalences of Morita type with level induced by bounded extensions and we prove the first part of Theorem~\ref{main-thm} (see Theorems~\ref{main-thm-semt} and \ref{main-thm-semt2}).  Section~\ref{Sect: Gorenstein} contains some criteria such that    a functor between module categories   induces  triangle functors between  stable categories of Gorenstein projective  modules and Gorenstein defect categories and their applications  to bounded extensions,  and we prove Criterion~\ref{main criteria} (see {Criteria~\ref{Crit: restr-Gproj2} and \ref{Crit: restr-Gproj3}) and the second part of  Theorem~\ref{main-thm} (see Theorems~\ref{main-thm-def1} and \ref{main-thm-def2}).
In Section~\ref{Sect:Conjecture}, we apply the results obtained in Sections~\ref{semt} and \ref{Sect: Gorenstein} to some homological conjectures. The last Section~\ref{Sect:Examples} contains many examples including  trivial extensions and triangular matrix algebras, which permits recovering and generalising many results in the literature including the results around the arrow removal operation \cite{EPS22, GPS21}.

\section{Bounded extensions}\label{Sect: notion}

In this section, we  introduce  the notion of   bounded extensions and provide some examples.

Let $B$ be a finite dimensional  algebra over a fixed field $\bfk$ and $M$ be a $B$-$B$-bimodule. Define $M^{\otimes_B0}=B$ and $M^{\otimes_B(i+1)}=M\otimes_B M^{\otimes_Bi }$ for $i\ge 0$. We say that $M$ is \emph{$B$-tensor nilpotent} if  $M^{\otimes_B p}=0$ for some $p\ge 1$.

\begin{Def}[{\cite{CLMS22erratum}}]\label{def-bound-exten} An extension $B\subset A$ of finite dimensional algebras is called bounded if
\begin{itemize}
\item[$(1)$] $A/B$ is $B$-tensor nilpotent;

\item[$(2)$] $A/B$ has finite projective dimension as a $B$-$B$-bimodule;

\item[$(3)$] $\Tor_i^B(A/B, (A/B)^{\otimes_B j})=0$ for all $i, j\geq 1$.

\end{itemize}
\end{Def}

\begin{Remark}\label{rem:equivalence-tor}
It follows from \cite [Proposition 2.4]{Lof76} that
$\Tor_i^B(A/B, (A/B)^{\otimes_B j})=0$ for all $i, j\geq 1$ if and only if $\Tor_i^B( (A/B)^{\otimes_B j}, A/B)=0$ for all $i, j\geq 1$.

%\item[$(2)$]  This definition of bounded extensions is not essentially new and  has in fact been suggested by  Cibils, Lanzilotta, Marcos, and Solotar in \cite{CLMS22erratum}.
%\end{itemize}
\end{Remark}
%
%\gd{Question: Give examples of extensions which satisfy our definition, but not the definition of Cibils et al.}
%

%\medskip
Now, we give some examples of extensions which satisfy our definition, but not the definition in \cite{CLMS22}, where it is  required that $A/B$ is projective as a left or right $B$-module.
More examples can be found in Section~\ref{Sect:Examples}.

Let $B$ be a finite dimensional  algebra and $M$ be  a finite dimensional $B$-$B$-bimodule.
 The trivial extension $A=B\ltimes M$ is by definition $B\oplus M$ and the multiplication is given
  by  $$(b, m)(b', m')=(bb', mb'+bm'), b, b'\in B, m, m'\in M.$$

The following result is trivial.
\begin{Prop}\label{Prop-bounded-exten}
 The trivial extension $B\subset A=B\ltimes M$ is a bounded extension if and only if $\pd( _{B}M_B)<\infty$, $M^{\otimes _B p}=0$ for some integer $p$ and
$\Tor _i^B(M, M^{\otimes _B j})=0$ for each $i,j\geq 1$.
%Hence, Theorems~\ref{main-thm-semt}, \ref{main-thm-semt2},    \ref{main-thm-def2}, \ref{main-cor-semt-1} and \ref{main-cor-semt-2}  apply.

%If this is the case,  there are triangle equivalences
%$$ \Sg(B) \simeq \Sg(A), \  \stgp{B}\simeq \stgp{A},\   \mbox{and} \   \Gdef(B)\simeq \Gdef(A).$$
%Moreover, $A$ and $B$ are singularly equivalent of Morita type with level.
\end{Prop}

We could produce many examples following the spirit of the above result.
\begin{Exam}\label{Ex-bounded-extensions}
Let $B$ be a finite dimensional  algebra. Let ${}_BU$ and $V_B$ be two modules of finite projective dimension, but not projective. Put  $M=U\ot_\bfk V$.
Then the trivial extension  $B\subset A=B\ltimes M$ is a bounded extension if and only if $\Tor^B_i(V, U)=0, \forall i\geq 0$.
\end{Exam}

In order to give more examples, let us recall basic notations about quivers and their representations.

Let $Q=(Q_0,Q_1)$ be a quiver with $Q_0$ the set of vertices and $Q_1$ the set of arrows between vertices. The source and target of an arrow $\alpha\in Q_1$ are denoted by $s(\alpha)$ and $t(\alpha)$, respectively. Note that the composition of two arrows $\alpha,\beta\in Q_1$ is written as $\beta\alpha$, where $\alpha$ comes first and then $\beta$ follows. Suppose that $Q$ is a finite acyclic quiver, that is, a quiver with finitely many vertices and finitely many arrows without loops or oriented cycles. Then the path algebra $\bfk Q$ of $Q$ is a finite dimensional $\bfk$-algebra of global dimension $1$. As usual, let $e_i$ be the primitive idempotent element in $\bfk Q$ corresponding to the vertex $i\in Q_0$, and let $S(i)$ and $T(i)$ stand for the simple left $B$-module and the simple right $B$-module corresponding to the vertex $i\in Q_0$, respectively.

The following is a  concrete example.
\begin{Exam}
Let $B$ be the following quiver algebra over $\bfk$
$$1\lraf{\alpha}2\lraf{\beta}3\lraf{\gamma}4.$$
%$$\xymatrix{
%1\ar[r]^{\alpha}&2\ar[r]^{\beta}&3\ar[r]^{\gamma}&4.
%}$$
%Then its opposite algebra $B^{\rm op}$ is the following quiver algebra
%$$\xymatrix{
%1^{\circ}
%&2^{\circ}\ar[l]_{a^{\circ}}
%&3^{\circ}\ar[l]_{b^{\circ}}
%&4^{\circ}\ar[l]_{c^{\circ}}
%}$$
%Let $S(i)$ and $P(i)$ be the simple and projective $B$-module corresponding to the vertex $i$ for $1\le i\le 4$, and let $T(i)$ and $Q(i)$ be the simple and projective $B^{\rm op}$-module corresponding to the vertex $i$ for $1\le i\le 4$.
Define $M:=S(3)\otimes_{\bfk} T(2)$ and $A:=B\ltimes M$. It is easy to check that
 $\Tor^B_i(T(2), S(3))=0$ for any $i\geq 0$, and
then it follows from
Example~\ref{Ex-bounded-extensions} that
$B\subset A$ is a bounded extension and $_BM\not\in \pmodcat{B}$ and $M_B\not\in \pmodcat{B^{\rm op}}$.
%Hence, the extension $B\subset A$ is a bounded extension in the sense of Definition~\ref{def-bound-exten}, but not  of \cite{CLMS22}.
\end{Exam}

Let's end this section by pointing out that  relation extensions, which are closely related to cluster-tilted algebras
\cite{ABS08, AGST16},
are not   bounded extensions in general. More precisely, let $B$ be an algebra of global dimension at most two and denote $DB=\Hom_\bfk(B, \bfk)$ with the canonical $B$-$B$-bimodule structure.
Let  $M=\Ext _B^2(DB,B)$ which has the  induced    $B$-$B$-bimodule structure.
The trivial extension $A=B\ltimes M$  is called the relation extension of $B$, see \cite{ABS08}. In particular, when the quiver of $B$ has no oriented cycles, $B$ is a tilted algebra if and only if  its \emph{relation extension} is cluster-tilted \cite{ABS08}.  Note that even if $B$ has global dimension at most two, $A$ has non vanishing singularity category in general  \cite{CGL15, Lu16}.

The following example, studied in \cite[Example 2.7]{ABS08}, \cite[Example 5.3]{AGST16} and \cite[Example 5.5]{CLMS22}, shows that the relation extension is not bounded in general.
\begin{Exam}
Let $B = \bfk Q/I$ where $Q$ is the quiver
$$
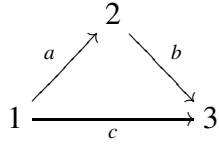
\xymatrix{& 2 \ar[dr]^b &  \\
1\ar[ur]^a \ar[rr]_c& & 3}$$
and $I=\langle ba \rangle$. Its relation extension $A$ is given by the quiver
$$
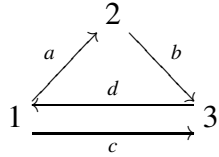
\xymatrix{& 2 \ar[dr]^b &  \\
1\ar[ur]^a \ar@<-1ex>[rr]_c& & 3\ar@<-1ex>[ll]_d}$$ with relations
$\{ba, ad, db, dcd\}$. As pointed out in \cite{CLMS22} (the paragraph before \cite[Lemma 5.14]{CLMS22}),
$A/B$ is not $B$-tensor nilpotent and then the relation extension $B\subset A$ is not bounded.
\end{Exam}

\section{Singular equivalences  of Morita type with level  induced by bounded extensions}\label{semt}

In this section, we are concerned about singular equivalences induced by bounded extensions.

%Let $A$ be a finite dimensional algebra over a field $\bfk$. Let $\Modcat{A}$ be the category of left $A$-modules, and $\modcat{A}$ be the category of finitely generated left $A$-modules. we denoted by $\D{A}$ the derived category of complexes over $A$, and by $\Db{\modcat{A}}$ the derived category of bounded complexes over $\modcat{A}$.
%Recall that a complex is called \emph{perfect} in $\Db{\modcat{A}}$ if it is isomorphic to a bounded complex of finite generated projective modules. Let $\per(A)$ be the full subcategory of $\Db{\modcat{A}}$ consisting of all perfect complexes, and $\Kb{\pmodcat{A}}$ be the full subcategory of $\Db{\modcat{A}}$ consisting of all bounded complexes of finite generated projective modules. For convenience, we do not distinguish between $\per(A)$ and $\Kb{\pmodcat{A}}$.

For our proof of Theorem \ref{main-thm}, we need the following three lemmas  whose easy proofs are left to the reader.
%The first lemma is well-known. Here, we give a proof for readers' convenience.
\begin{Lem}\label{lem-tensor}
Let $A$, $B$ and $C$ be finite dimensional algebras, $M$ be an $A$-$B$-bimodule and $N$ be a $B$-$C$-bimodule.
If $\Tor _i^B(M,N)=0$ for any $i>0$, then $M\otimes _B^{\mathbb{L}}N\cong M\otimes _BN$ in $\Db{\modcat{A\otimes_{\bfk}C^{\rm op}}}$.
\end{Lem}
%\begin{proof}
%Let $$\cdots \longrightarrow P^{-2} \longrightarrow
%P^{-1} \longrightarrow P^{0} \longrightarrow
%M \longrightarrow 0 $$
%be the projective resolution of $M$ as an $A$-$B$ bimodule
%and let $P^\bullet$ be the deleted complex $$ \cdots \longrightarrow P^{-2} \longrightarrow
%P^{-1} \longrightarrow P^{0} \longrightarrow 0.$$ Then
%$M\otimes_B^{\mathbb{L}}N\cong P^\bullet\otimes _BN$ in $\Db{A\otimes_{\bfk} C^{op}}$. Since $\Tor_i^B(M,N)=0$ for any $i>0$, the complex
%of $A$-$C$-bimodules
%$$\cdots \longrightarrow P^{-2}\otimes _BN \longrightarrow
%P^{-1}\otimes _BN \longrightarrow P^{0}\otimes _BN \longrightarrow
%M\otimes _BN \longrightarrow 0$$
%is exact, and thus $M\otimes _BN\cong P^\bullet\otimes _BN$ in $\Db{A\otimes_{\bfk} C^{op}}$.
%Therefore, $M\otimes _B^{\mathbb{L}}N\cong M\otimes _BN$ in $\Db{A\otimes_{\bfk} C^{op}}$.
%\end{proof}

%The following lemma is essentially due to \cite[Lemma 2.5]{AKLY17}.
\begin{Lem}\label{lem-preserve-kbproj}
Let $B\subset A$ be an extension of finite dimensional algebras and denote $M=A/B$.
Then the following statements hold:
\begin{itemize}
\item[$(1)$] The functor $A\otimes _B^{\mathbb{L}}-\otimes _B^{\mathbb{L}}A :\Db{\modcat{B^e}}\rightarrow \Db{\modcat{A^e}}$
sends $\Kb{\pmodcat{B^e}}$ to $\mathrm{K}^b(\pmodcat{A^e})$.

\item[$(2)$]  If $\pd (_BM_B)<\infty$, then the functor $_BM\otimes _B^{\mathbb{L}}-:\Db{\modcat{B^e}}\rightarrow \Db{\modcat{B^e}}$
sends $\Kb{\pmodcat{B^e}}$ to $\Kb{\pmodcat{B^e}}$.
\end{itemize}
\end{Lem}

%\begin{proof}
%(1) This follows from the isomorphisms $$A\otimes _B^{\mathbb{L}}(B\otimes _{\bfk}B)\otimes_B^{\mathbb{L}}A
%\cong (A\otimes_B^{\mathbb{L}}B)\otimes_{\bfk}(B\otimes_B^{\mathbb{L}}A)\cong A\otimes _{\bfk}A$$ and
%\cite[Lemma 2.5]{AKLY17}.
%
%(2) Let $P^\bullet \rightarrow A/B$ be the bounded projective resolution of $A/B$ as a $B$-$B$-bimodule.
%Then $_B(A/B)\otimes_B^{\mathbb{L}}(B\otimes_{\bfk}B)\cong P^\bullet \otimes _B(B\otimes _{\bfk}B)$ in $\Db{B^e}$.
%Since $(B\otimes_{\bfk}B)\otimes_B(B\otimes_{\bfk}B)\cong B\otimes_{\bfk}B\otimes_{\bfk}B\cong (B\otimes_{\bfk}B)^{\dim_{\bfk}B}$,
%we have that $P^i \otimes _B(B\otimes _{\bfk}B)\in \pmodcat{B^e}$. Therefore, $P^\bullet \otimes _B(B\otimes _{\bfk}B)
%\in \Kb{\pmodcat{B^e}}$ and then the statement follows from \cite[Lemma 2.5]{AKLY17}.
%\end{proof}

\begin{Lem}[{compare with \cite[Section 4]{CL20}}]\label{lem-ring-extension-tor}
Let $B\subset A$ be an extension of finite dimensional algebras.
If $\Tor _i^B(A/B, (A/B)^{\otimes _Bj})=0$ for each $i,j\geq 1$, then the following statements hold:
\begin{itemize}
\item[$(1)$] $\Tor _i^B(A, A)=0$ for each $i\geq 1$;

\item[$(2)$] $\Tor _i^B((A/B)^{\otimes _Bj}, A)=0$ for each $i,j\geq 1$;

\item[$(3)$] $\Tor _i^B(A, (A/B)^{\otimes _Bj}\otimes _BA)=0$ for each $i,j\geq 1$.

\end{itemize}
\end{Lem}

%\begin{proof}
%(1) For any $i\geq 1$, applying $(A/B)\otimes_B-$ to the exact sequence of left $B$ modules
%$$0\rightarrow B\rightarrow A\rightarrow A/B\rightarrow 0$$
%and using the fact that $\Tor _i^B(A/B, A/B)=0$,
%we have that
% $\Tor _i^B(A/B, A)=0$. Similarly, applying $-\otimes _BA$ to the exact sequence of right $B$ modules
%$$0\rightarrow B\rightarrow A\rightarrow A/B\rightarrow 0,$$ we get $\Tor _i^B(A, A)=0$ for each $i\geq 1$.
%
%(2) It follows from \cite[Corollary 4.3]{CL20} that $\Tor _i^B((A/B)^{\otimes _Bj},A/B)=0$ for each $i,j\geq 1$.
%Applying $(A/B)^{\otimes _Bj}\otimes _B-$ to the exact sequence of left $B$ modules
%$$0\rightarrow B\rightarrow A\rightarrow A/B\rightarrow 0,$$ we infer that $\Tor _i^B((A/B)^{\otimes _Bj}, A)=0$ for each $i,j\geq 1$.
%
%(3) For each $i,j\geq 1$, we have that $\Tor _i^B(A/B, (A/B)^{\otimes _Bj}\otimes _BA)=0$
%by (2) and \cite[Lemma 4.2]{CL20}.
%Applying $-\otimes _B(A/B)^{\otimes _Bj}\otimes _BA$ to the exact sequence of right $B$ modules
%$$0\rightarrow B\rightarrow A\rightarrow A/B\rightarrow 0,$$ we infer that $\Tor _i^B(A, (A/B)^{\otimes _Bj}\otimes _BA)=0$ for each $i,j\geq 1$.
%\end{proof}

The following two lemmas will be crucial in the proof of Theorem~\ref{main-thm-semt}.
\begin{Lem}\label{Lem: compactness}
 Let $B\subset A$ be an extension of finite dimensional algebras and denote $M=A/B$.  Assume that $\pd (_BM_B)<\infty$ and  $\Tor_i^B(M, M^{\otimes_B j})=0$ for all $i, j\geq 1$.
Then  $A\otimes _BM^{\otimes _Bj}\otimes _BA \in
\Kb{\pmodcat{A^e}}$ for any $j\geq 1$.
\end{Lem}

\begin{proof}
  Since $\Tor _i^B(M, M^{\otimes _Bj})=0$ for each $i,j\geq 1$,
it follows from Lemma~\ref{lem-tensor} that $$M\otimes _B
M\otimes _B\cdots \otimes _B
M\cong M\otimes _B^{\mathbb{L}}
M\otimes _B^{\mathbb{L}}\cdots \otimes _B^{\mathbb{L}}
M \in \Db{\modcat{B^e}},$$ which belongs to $\Kb{\pmodcat{B^e}}$ by $\pd (_BM_B)<\infty$ and by Lemma~\ref{lem-preserve-kbproj} (2).
Therefore, $M^{\otimes _Bj}\in \Kb{\pmodcat{B^e}}$ for any $j\geq 1$,
and it follows from Lemma~\ref{lem-preserve-kbproj} (1) that  $$A\otimes _B^{\mathbb{L}}M^{\otimes _Bj}\otimes _B^{\mathbb{L}}A \in
\Kb{\pmodcat{A^e}}.$$  On the other hand, Lemma~\ref{lem-ring-extension-tor} (2), Lemma~\ref{lem-ring-extension-tor} (3)
and Lemma~\ref{lem-tensor} yield that $$A\otimes _B^{\mathbb{L}}M^{\otimes _Bj}\otimes _B^{\mathbb{L}}A
\cong A\otimes _BM^{\otimes _Bj}\otimes _BA.$$ Therefore, $A\otimes _BM^{\otimes _Bj}\otimes _BA \in
\Kb{\pmodcat{A^e}}$ for any $j\geq 1$.
\end{proof}

\begin{Lem}[{\cite[Theorem 3.1 and Proposition 3.7]{OPS19}}]\label{Lem: restriction to singularity}
Let $B\subset A$ be an   extension of finite dimensional algebras.
If $\pd (A_B)<\infty$ and $\pd({}_BA)<\infty$, then there is an adjoint pair
$${_A}A\otimes^{\mathbb{L}}_{B}-: \Sg(B)\rightleftharpoons\Sg(A):  \Res^A_B=_BA\otimes _A-. $$
Assume, moreover, that $\pd({_{B}M_B})<\infty$ with $M=A/B$, then ${_A}A\otimes^{\mathbb{L}}_{B}-: \Sg(B)\to \Sg(A)  $ is fully faithful.
\end{Lem}

\begin{Remark}\label{Rem: Restrition to singularity}
Let us remark on the proof of the first statement of the above result.
Consider the adjoint triple between derived categories
$$\xymatrix@!=8pc{ \D{A} \ar[r]|{\Res^A_B=_BA\otimes _A-} & \D{B} \ar@<-3ex>[l]|{A\otimes _B^{\mathbb{L}}-}
\ar@<+3ex>[l] }.$$
Since $\pd (A_B)<\infty$ and $\pd({}_BA)<\infty$,    this adjoint triple extends one step upwards and one step downwards, respectively.
In this case,  by \cite[Lemma 2.5]{AKLY17}, the three functors: the left adjoint of $A\otimes _B^{\mathbb{L}}-$, $A\otimes _B^{\mathbb{L}}-$ and $\Res^A_B$ restrict to $\Kb{\rm proj}$, by \cite[Lemma 2.8]{AKLY17} and \cite[Lemma 3.2]{OPS19}, the three functors:   $A\otimes _B^{\mathbb{L}}-$,  $\Res^A_B$ and the right adjoint of $\Res^A_B$ restrict to $\Db{\rm mod}$, hence, both  $A\otimes _B^{\mathbb{L}}-$ and   $\Res^A_B$ restrict to $\Sg$
(see \cite[Lemma 1.2]{O04}).

This kind of restriction results are rather useful.
\end{Remark}

Now we will  show that for a bounded extension $A\subset B$, the adjoint pair is an adjoint equivalence. Furthermore, this singular equivalence is of special form, singular equivalence of Morita type with level in the sense of Wang \cite{Wang15}.
%This constitutes the first part of Theorem~\ref{main-thm} in the
%introduction.

\begin{Thm}\label{main-thm-semt}
Let $B\subset A$ be a bounded  extension. Then
$${_A}A\otimes^{\mathbb{L}}_{B}-: \Sg(B)\rightleftharpoons\Sg(A):  \Res^A_B $$ is an adjoint equivalence which induces a singular equivalence of Morita type with level between $A$ and $B$.
\end{Thm}
\begin{proof}
Denote $M=A/B$. Since $\pd({}_BM_B)<\infty$, we get that $\pd (M_B)<\infty$ and $\pd ({}_BM)<\infty$.
Now the exact sequence of $B$-$B$-bimodules
 $0\rightarrow B\rightarrow A\rightarrow M\rightarrow 0$  shows that $\pd (A_B)<\infty$ and $\pd({}_BA)<\infty$. By   Lemma~\ref{Lem: restriction to singularity},
   there is an adjoint pair
$$ {_A}A\otimes^{\mathbb{L}}_{B}-: \Sg(B)\rightleftharpoons\Sg(A):  \Res^A_B $$and   ${_A}A\otimes^{\mathbb{L}}_{B}-: \Sg(B)\to \Sg(A)  $ is fully faithful.

What we add here is that $${_A}A\otimes^{\mathbb{L}}_{B}-: \Sg(B)\rightleftharpoons\Sg(A):  \Res^A_B $$ is an adjoint equivalence which are in fact singular equivalences of Morita type with level.
 By \cite[Theorem 3.6]{Dal21} and \cite[Proposition 3.3]{Qin22}
it suffices to show that
the natural morphisms $B\rightarrow GF(B)$ and $FG(A)\rightarrow A$
are isomorphisms in $\Sg(B^e)$ and $\Sg(A^e)$, respectively, where
$F={_A}A\otimes^{\mathbb{L}}_{B}-$ and $G=\Res^A_B$.

We observe that the mapping cone of $B\rightarrow GF(B)$ is $M$, and
the assumption
$\pd( _BM_B)<\infty$ yields that $B\cong GF(B)$ in $\Sg(B^e)$.This  also follows from the fact that   ${_A}A\otimes^{\mathbb{L}}_{B}-: \Sg(B)\to \Sg(A)  $ is fully faithful.

By Lemma~\ref{lem-ring-extension-tor} (1) and Lemma~\ref{lem-tensor},
we obtain that
$$FG(A)=A\otimes _B^{\mathbb{L}}A\cong A\otimes _BA \in \Db{\modcat{A^e}},$$
and then the mapping cone of $FG(A)\rightarrow A$ is isomorphic
to the complex $$0\rightarrow A\otimes _BA\stackrel{\mu}{\rightarrow} A\rightarrow 0$$
induced by the product $\mu$ of $A$. Since $M^{\otimes _Bp}=0$ for some integer $p$, it follows from \cite[Proposition 2.3]{CLMS21b}
that there is a long exact sequence
of $A$-$A$-bimodules
$$\xymatrix@C=0.8pc{ 0 \ar[r] & A\otimes _BM^{\otimes _B{(p-1)}}\otimes _BA \ar[r] & \cdots
\ar[r] &  A\otimes _BM\otimes _BA \ar[r] &  A\otimes _BA \ar[r]^-{\mu} & A \ar[r] & 0.}$$
Therefore, % the nature chain map
%$$\xymatrix@C=0.8pc{ 0 \ar[r] & A\otimes _B(M)^{\otimes _B{(p-1)}}\otimes _BA \ar[r] & \cdots
%\ar[r] &  A\otimes _B(M)\otimes _BA \ar[r] \ar[d]& 0 \ar[d] \ar[r] &0\\
%& & 0 \ar[r] & A\otimes _BA \ar[r] & A \ar[r] &0}$$
%is a quasi-isomorphism, and thus
the mapping cone of $FG(A)\rightarrow A$ is isomorphic
to the complex
$$\xymatrix@C=0.8pc{ 0 \ar[r] & A\otimes _BM^{\otimes _B{(p-1)}}\otimes _BA \ar[r] & \cdots
\ar[r] &  A\otimes _BM\otimes _BA \ar[r] & 0}$$
in $\Db{\modcat{A^e}}$. Now  by Lemma~\ref{Lem: compactness},   $A\otimes _BM^{\otimes _Bj}\otimes _BA \in
\Kb{\pmodcat{A^e}}$ for any $j\geq 1$, and thus $FG(A)\cong A$
in $\Sg(A^e)$.
\end{proof}

\begin{Cor}For a bounded extension $B\subset A$, $A$ has finite global dimension if and only if so does $B$.

\end{Cor}
\begin{Remark}
A refinement of the above corollary was proved by  Cibils, Lanzilotta, Marcos, and Solotar for left or right bounded extensions; see
\cite[Theorems 4.2 and 4.3]{CLMS22}.
\end{Remark}
A functor $F:\mathcal{B}\rightarrow \mathcal{C}$ between abelian categories
is called a   $t$-\textit{eventually homological isomorphism}
if there is an integer $t$ such that
for every $j > t$, the induced homomorphism
 $\Ext ^j _{\mathcal{B}}(X, Y ) \to \Ext ^j _{\mathcal{C}}(F(X), F(Y) )$ is an isomorphism
for all objects $X, Y\in \mathcal{B}$.
This notion was introduced in \cite{PSS14} and has been  used widely
to reduce homological
properties of algebras \cite{ GPS21, PSS14, QS23}.

\begin{Thm} \label{main-thm-semt2}
Let $B\subset A$ be a bounded extension. Then the functor $\Res^A_B:\modcat{A}\rightarrow \modcat{B}$ is a $t$-eventually homological isomorphism
for some integer $t$.
\end{Thm}
\begin{proof}   Let $F={_A}A\otimes^{\mathbb{L}}_{B}-$ and $G=\Res^A_B$.
Let $\eta : FG\rightarrow 1_{\D{A}}$ be the counit of the adjoint pair
$$\xymatrix@C=2pc{\D{B} \ar@<+1ex>[rr]^{F}
&&\D{A}\ar@<+1ex>[ll]^{G}}.$$
For any $X,Y\in \modcat{A}$, applying $\Hom _{\D{A}}(-, Y[i])$ to  the  canonical triangle
$$\xymatrix{FG(X) \ar[r]^(0.63){\eta _X} & X
 \ar[r] & \Cone (\eta _X) \ar[r] & FG(X)[1]}$$
in $\D{A}$,
 we get
a long exact
sequence
$$\xymatrix@C=0.5pc{\cdots \ar[r] & \Hom _{\D{A}}(\Cone (\eta _X), Y[i]) \ar[r] &
\Hom _{\D{A}}(X, Y[i]) \ar[r] & \Hom _{\D{A}}(FG(X), Y[i])\ar[r] & \cdots  (\star).
}$$
Since $FG(X)\cong FG(A)\otimes ^{\mathbb{L}} _AX$, it follows from the commutative diagram
 $$\xymatrix@C=2pc{FG(X) \ar[r]^{\eta _X} \ar[d]^{\cong} & X
\ar[d]^{\cong} \ar[r] & \Cone (\eta _X) \ar[r]\ar[d] &FG(X)[1] \ar[d]^{\cong}
\\ FG(A)\otimes ^{\mathbb{L}}_AX \ar[r]^(0.6){\eta _A\otimes ^{\mathbb{L}}_AX }  & A\otimes ^{\mathbb{L}} _AX \ar[r]  & \Cone (\eta _A)\otimes ^{\mathbb{L}}_AX \ar[r] & FG(A)\otimes ^{\mathbb{L}}_AX[1]}$$
that $\Cone (\eta _X)\cong \Cone (\eta _A)\otimes ^{\mathbb{L}}_AX $ in $\D{A}$.
On the other hand, it follows from the proof of Theorem \ref{main-thm-semt} that
$\Cone (\eta _A)\in\Kb{\pmodcat{A^e}}$.
Hence, there is a bounded complex of projective $A$-$A$-bimodules
$$P^\bullet: 0\rightarrow P^{-t}\longrightarrow \cdots
\longrightarrow P^{s} \longrightarrow
0 $$ such that $P^\bullet \cong \Cone (\eta _A)$ in $\D{A^e}$.
Therefore, the complex $\Cone (\eta _X)\cong P^\bullet \otimes _AX $ in $\D{A}$
is of the form $$0\rightarrow P^{-t}\otimes _AX \longrightarrow \cdots
\longrightarrow P^{s}\otimes _AX \longrightarrow
0 ,$$ where each $_AP^{i}\otimes _AX\in \proj(A)$ since $P^{i}\in \proj(A^e)$.
It follows that $$\Hom _{\D{A}}(\Cone (\eta _X), Y[i])\cong\Hom _{\D{A}}(P^\bullet \otimes _AX, Y[i])\cong \Hom _{\mathrm{K}(A)}(P^\bullet \otimes _AX, Y[i]),$$
which is equal to zero for any $i>t$. Therefore, the long exact sequence $(\star)$
yields that $$\Hom _{\D{A}}(X, Y[i]) \cong \Hom _{\D{A}}(FG(X), Y[i]) \cong \Hom _{\D{B}}(G(X), G(Y)[i])$$
for any $i>t$. Therefore, $\Ext _A^i(X,Y)\cong \Ext _B^i(G(X),G(Y))$ for any $i>t$
and any $X,Y\in \modcat{A}$. This shows that $G$ is a $t$-eventually homological isomorphism.
\end{proof}

\section{The stable categories of Gorenstein
projective modules and the Gorenstein defect categories
for bounded  extensions}\label{Sect: Gorenstein}

In this section,  we will compare the stable categories of Gorenstein
projective modules and the Gorenstein defect categories
between two algebras linked by a bounded  extension.

Throughout this section, let  $\Lambda$ and $\Gamma$ be two   finite dimensional $\bfk$-algebras and  let  $X$ be a finite dimensional $\Gamma$-$\Lambda$-bimodule such that $X$ has finite projective dimension both as left $\Gamma$-module and as right $\Lambda$-module.

We will consider the following question which generalises the situation considered in \cite{OPS19}:
\begin{Ques}\label{Ques: fundamental}
     When does the functor $X\otimes_\Lambda-:\modcat{\Lambda} \rightarrow \modcat{\Gamma}$  induce a triangle functor from $\stgp{\Lambda}$ to $ \stgp{\Gamma}$?
\end{Ques}
Some remarks are in order. Following the spirit of  Remark~\ref{Rem: Restrition to singularity}, under the condition that  $\pd(_\Gamma X), \pd(X_\Lambda)<\infty$, the adjoint pair between derived categories
$$\xymatrix@!=12pc{ \D{\Lambda} \ar@<+1.5ex>[r]|{X\otimes_\Lambda^{\bbL}- } & \D{\Gamma}
\ar@<+1.5ex>[l]|{\bbR\Hom_\Gamma(X,-) }} $$
  extends one step upwards and one step  downwards and we obtain an adjoint quadruple:
$$\xymatrix@!=12pc{ \D{\Lambda} \ar@<+1.5ex>[r]|{X\otimes_\Lambda^{\bbL}-\simeq \bbR\Hom(X^{\tr_{\Lambda^{\op}}},-)}
\ar@<-4.5ex>[r]|{\bbR\Hom_\Lambda( X^{\tr_{\Gamma}},-)}
& \D{\Gamma} \ar@<-4.5ex>[l]|{X^{\tr_{\Lambda^{\op}}}  \otimes _\Gamma^{\bbL}-}
\ar@<+1.5ex>[l]|{\bbR\Hom_\Gamma(X,-)\simeq X^{\tr_{\Gamma}}  \otimes _\Gamma^{\bbL}- }}, $$
where $$X^{\tr_{\Gamma}}=\bbR\Hom_\Gamma(X,\Gamma), \  X^{\tr_{\Lambda^{\op}}}=\bbR\Hom_{\Lambda^{\op}}(X,\Lambda).$$
 By \cite[Lemma 2.5]{AKLY17}, the two functors  $X^{\tr_{\Lambda^{\op}}}  \otimes _\Gamma^{\bbL}-$ and $X\otimes_\Lambda^{\bbL}-\simeq \bbR\Hom(X^{\tr_{\Lambda^{\op}}},-) $ restrict to $\Kb{\rm proj}$, by \cite[Lemma 2.8]{AKLY17} and \cite[Lemma 3.2]{OPS19} (or more generally for differential graded algebras, \cite[Proposition 2.5]{JYZ23}) the two functors $X\otimes_\Lambda^{\bbL}-\simeq \bbR\Hom(X^{\tr_{\Lambda^{\op}}},-) $ and  $ \bbR\Hom_\Gamma(X,-)\simeq X^{\tr_{\Gamma}}  \otimes _\Gamma^{\bbL}- $ restrict to $\Db{\rm mod}$, hence, $X\otimes_\Lambda^{\bbL}-\simeq \bbR\Hom(X^{\tr_{\Lambda^{\op}}},-) $  restricts to $\Sg$
(see \cite[Lemma 1.2]{O04}).

We give several criteria in  order to answer Question~\ref{Ques: fundamental}.

The first criterion is equivalent to asking that the adjoint quadruple can be extended  further one step downwards.

\begin{Crit}[{Compare with \cite[Propisition 3.5]{OPS19}}] \label{Crit: restr-Gproj0}    If  $X^{\tr_{\Gamma}}=\mathbb{R}\Hom_\Gamma(X,\Gamma)$ is perfect as complex of $\Lambda$-modules, then  the triangle functor
 $X\otimes_\Lambda^{\bbL}-\simeq \bbR\Hom(X^{\tr_{\Lambda^{\op}}},-) $  restricts to  the stable categories of Gorenstein
projective modules.
\end{Crit}

 \begin{proof}
This can be deduced from Hu and Pan  \cite{HP17} and also is made explicit in \cite{OPS19} in case that  there is  a ring homomorphism $f: \Lambda\to \Gamma$ and $X={}_\Gamma \Gamma_\Lambda$.

In fact, the condition that $X^{\tr_{\Gamma}}=\mathbb{R}\Hom_\Gamma(X,\Gamma)$ is perfect as complex of $\Lambda$-modules implies that
the adjoint quadruple:
$$\xymatrix@!=12pc{ \D{\Lambda} \ar[r]|{X\otimes_\Lambda^{\bbL}-\simeq \bbR\Hom(X^{\tr_{\Lambda^{\op}}},-)}
\ar@<-6ex>[r]|{\bbR\Hom_\Lambda( X^{\tr_{\Gamma}},-)}
& \D{\Gamma} \ar@<-3ex>[l]|{X^{\tr_{\Lambda^{\op}}}  \otimes _\Gamma^{\bbL}-}
\ar@<+3ex>[l]|{\bbR\Hom_\Gamma(X,-)\simeq X^{\tr_{\Gamma}}  \otimes _\Gamma^{\bbL}- }}$$
can be extended further one step downwards, because
$\bbR\Hom_\Lambda( X^{\tr_{\Gamma}},-)\simeq ( X^{\tr_{\Gamma}})^{\tr_{\Lambda}}\ot_{\Lambda}^\bbL-$ has a right adjoint
$\bbR\Hom_{\Lambda}((X^{\tr_{\Gamma}})^{\tr_{\Lambda}}, -)$.
The functor $$G:=\bbR\Hom_\Gamma(X,-)\simeq X^{\tr_{\Gamma}}  \otimes _\Gamma^{\bbL}-: \Db{\modcat{\Gamma}}\to \Db{\modcat{\Lambda}}$$ sends $\Kb{\pmodcat{\Gamma}}$ to $\Kb{\pmodcat{\Lambda}}$. Moreover, since $\pd(X_\Lambda)<\infty$, a suitable shift of  $F:=X\otimes_\Lambda^{\bbL}-$ is nonnegative in the sense of \cite[Definition 4.1]{HP17}.  By \cite[Theorem 5.3]{HP17}, $F$ restricts to the stable categories of Gorenstein
projective modules.
\end{proof}
The condition that  $X^{\tr_{\Gamma}}=\mathbb{R}\Hom_\Gamma(X,\Gamma)$ is perfect as complex of $\Lambda$-modules is somehow    strong. We will introduce two alternative conditions in the sequel; see Criteria~\ref{Crit: restr-Gproj2} and \ref{Crit: restr-Gproj3}.

  %  As we will see soon, it would be better to consider a little larger subcategory of the bounded derived category or the singularity category instead of  the  stable  category of Gorenstein projective modules.
We need several notations.
For an algebra $A$, we denote by $\Omega _A (-)$ the syzygy
endofunctor of the stable module category of $A$. Define
$${}^{\perp}A:=\{X\in \modcat{A} \ |\ \Ext_A^i(X, A)=0\mbox{ for all }i>0\}.$$
Notice that a module  in this subcategory is called \emph{semi-Gorenstein projective} by Ringel and Zhang \cite{RZ20}.
  Denote
$${\rm Gproj}(A)^{\bot}:=\{X\in \modcat{A}\mid \Ext_{A}^{i}(U,X)=0 \mbox{ for all } U\in \gpmodcat{A},\, i\ge 1\}.$$
%Recall that  $\fGd{A}$ is the full subcategory of
%$\Db{\modcat{A}}$ formed by those complexes quasi-isomorphic to bounded complexes of Gorenstein projective modules.
%Here, the definition of $\fGd{A}$ agrees with that in \cite{Kato02},
%where the objects in $\fGd{A}$ are called complexes of finite Gorenstein projective dimension,
%see \cite[Definition 2.7 and Proposition 2.10]{Kato02}.
%The following triangle equivalence
%$$ \stgp{A} \simeq \fGd{A}/\Kb{\pmodcat{A}}$$
%is well known, see, for instance,  \cite[Theorem 4.4.1]{Buch21},
% \cite[Theorem 4]{PZ15},  \cite[Lemma 4.1]{CR20}, or in a much more general setup \cite[Theorem 3.7]{YZ22}.
%It will be this category $\fGd{A} $ on which we will mainly work on.

%We begin with an observation which follows from  \cite[Proposition 3.7]%%{OPS19}.
%\begin{Prop}
%Let $B\subset A$ be a bounded extension. Then ${}_AA\otimes_B-: \stgp{B}\ra  \stgp{A}$ is fully faithful.
%\end{Prop}
%Now we will split the second part of Theorem~\ref{main-thm}  presented in %the
%introduction  into two results.

The second criterion concerns a right exact functor between module categories which,  under certain conditions,  sends  Gorenstein projective modules to  Gorenstein projective modules.
\begin{Crit}\label{Crit: restr-Gproj1}
  If  ${}_\Gamma X$ is projective and   $\Hom_{\Gamma}(X,\Gamma) \in {\rm Gproj}(\Lambda)^{\bot}$, then $X\otimes _\Lambda-:\modcat{\Lambda} \rightarrow \modcat{\Gamma}$   sends Gorenstein projective modules to Gorenstein projective modules and $X\otimes_\Lambda U\cong X\otimes_\Lambda^{\mathbb{L}} U \in \Db{\modcat{\Gamma}}$ for  any $U\in {\rm Gproj}(\Lambda)$. In this case,   there is  an induced triangle functor
$ X\otimes _\Lambda-:  \stgp{\Lambda} \rightarrow \stgp{\Gamma}$.
\end{Crit}

\begin{proof}
Let $U\in {\rm Gproj}(\Lambda)$. By definition, there is a complete  resolution $\cpx{P}$ of $U$.
It follows from $X\in \pmodcat{\Gamma}$ and $P^i\in \pmodcat{\Lambda}$ that $_{\Gamma}X\otimes_\Lambda P_i\in \pmodcat{\Gamma}$ for each $i$. As $\cpx{P}$ is exact, we see that $\rmH^i(X\otimes_\Lambda \cpx{P})=\Tor_{j-i}^{\Lambda}(X, B^{j+1}(\cpx{P}))$ for any $j>i$. Since the projective dimension of $X$ as right $\Lambda$-module is assumed to be finite, we can always choose $j$ sufficiently big so that this $\Tor$ group vanishes. Thus $X\otimes_\Lambda \cpx{P}$ is exact, which also shows that $\Tor_{i}^A(X, U)=0$ for $i>0$. By Lemma \ref{lem-tensor}, ${}_\Gamma X\otimes^\mathbb{L}_\Lambda U\cong {}_\Gamma X\otimes_\Lambda U \in  \Db{\Lambda}$ for $U\in {\rm Gproj}(\Lambda)$.
Note that $$\Hom_\Gamma(X\otimes_\Lambda \cpx{P}, \Gamma)\cong \Hom_\Lambda(\cpx{P},\Hom_\Gamma(X, \Gamma)).$$
It follows from $\Hom_{\Gamma}(X,\Gamma) \in {\rm Gproj}(\Lambda)^{\bot}$ that $\Hom_\Lambda(\cpx{P},\Hom_\Gamma(X, \Gamma))$ and $\Hom_\Gamma(X\otimes_\Lambda \cpx{P}, \Gamma)$ are  exact.
Thus $X\otimes_\Lambda \cpx{P}$ is a complete resolution of $X\ot_\Lambda U$, and therefore $X\ot_\Lambda U$ is a Gorenstein projective $\Gamma$-module.

For the second statement, it remains to show that the additive functor $ X\otimes _\Lambda-:  \stgp{\Lambda} \rightarrow \stgp{\Gamma}$ is a triangle functor.
Indeed, due to ${}_\Gamma X\otimes^\mathbb{L}_\Lambda U\cong {}_\Lambda X\otimes_\Gamma U \in  \Db{\Lambda}$ for $U\in {\rm Gproj}(\Lambda)$,
  the functor ${}_\Gamma X\otimes^\mathbb{L}_\Lambda-$ induces a triangle functor from $\fGd{\Gamma}$ to $\fGd{\Lambda}$. We have thus a commutative diagram
$$\xymatrix{\stgp{\Lambda}\ar[r]^{{}_\Gamma X\otimes_\Lambda-}\ar[d]^{\simeq}& \stgp{\Gamma}\ar[d]^{\simeq}\\
\fGd{\Lambda}/\Kb{\pmodcat{\Lambda}}\ar[d]\ar[r]^{{}_\Gamma X\otimes_\Lambda ^\mathbb{L}-} & \ar[d]\fGd{\Gamma}/\Kb{\pmodcat{\Gamma}}\\
\Sg(\Lambda)\ar[r]^{{}_\Gamma X\otimes_\Lambda ^\mathbb{L}-} &\Sg(\Gamma) }$$
Thus $ X\otimes _\Lambda-:  \stgp{\Lambda} \rightarrow \stgp{\Gamma}$ is a triangle functor.
\end{proof}

 \begin{Remark}Although Criterion~\ref{Crit: restr-Gproj1} is somehow trivial, it generalises \cite[Proposition 3.4]{OPS19} in which the authors consider the situation where there is  a ring homomorphism $f: \Lambda\to \Gamma$ and $X={}_\Gamma \Gamma_\Lambda$.

\end{Remark}

\begin{Remark} In the hypotheses of Criterion~\ref{Crit: restr-Gproj1}, if we assume that   ${}_\Gamma X$ is projective, then  $\Hom_{\Gamma}(X,\Gamma)$ $\in {\rm Gproj}(\Lambda)^{\bot}$ is equivalent to saying that   $X\otimes_\Lambda U$ belongs to ${}^\perp\Gamma$ whenever   $U\in {\rm Gproj}(\Lambda)$. In fact, for  $G\in {\rm Gproj}(\Lambda)$ and a deleted projective   resolution $\cpx{Q}$ of $G$,
   $X\otimes_\Lambda \cpx{Q}$ is exact which is a projective resolution of $X\otimes_\Lambda G$. By the isomorphism of complexes $$\Hom_\Lambda(\cpx{Q}, \Hom_{\Gamma}(X,\Gamma))\cong \Hom_\Gamma(X\otimes_\Lambda\cpx{Q},  \Gamma),$$
   for any $i>0$, there is an  isomorphism
$$\Ext_\Lambda^i(G, \Hom_{\Gamma}(X,\Gamma))\cong \Ext_\Gamma^i(X\otimes_\Lambda G,  \Gamma).$$
The equivalence deduces from the above isomorphisms.
\end{Remark}
The  following example verifies the hypothesis of Criterion~\ref{Crit: restr-Gproj1}, but not that of Criterion~\ref{Crit: restr-Gproj0}.
\begin{Exam}\cite[2 Example]{HKR08}
Let $\Gamma$ be a finite dimensional \bfk-algebra and $\Lambda=kQ/I$ be the algebra where $Q$ is the quiver
$$\xymatrix{1   \ar@(ul,dl)_\beta & 2 \ar[l]_\alpha }$$
and $I=\langle \beta ^2,  \beta \alpha \rangle$. Let $X=\Gamma \otimes_{\bfk}  e_1\Lambda $. It is clearly that $_\Gamma X_\Lambda$ is a projective bimodule, thus  $_\Gamma X, X_\Lambda$ are projective,  and $\Hom_{\Gamma}(X, \Gamma) \cong \Hom_\bfk(e_1\Lambda, \bfk) \otimes_{\bfk} \Gamma \in {\rm Gproj}(\Lambda)^{\bot}.$ But $\Hom_\bfk(e_1\Lambda, \bfk)$ has infinite projective dimension as left $\Lambda$-module.  Hence, $\bbR\Hom_{\Gamma}(X, \Gamma)= \Hom_{\Gamma}(X, \Gamma)$ falls into ${\rm Gproj}(\Lambda)^{\bot}$, but not perfect as $\Lambda$-module.

\end{Exam}

\begin{Lem}\label{Lem: Gp orthogonal}
Let $\Sigma$ be a finite dimensional $\bfk$-algebra and $M\in \modcat{\Sigma}$. If $M$ is  quasi-isomorphic to a bounded complex with each term in $ {\rm Gproj}(\Sigma)^{\bot}$, then $M\in  {\rm Gproj}(\Sigma)^{\bot}$.
\end{Lem}

\begin{proof}
Let $M$ quasi-isomorphic to
$$\cpx{L}: 0\ra L^{s} \ra L^{s+1}\ra \cdots \ra L^{t}\ra 0  $$
with $L^j \in {\rm Gproj}(\Sigma)^{\bot}, \forall j$.
For any $U\in {\rm Gproj}(\Sigma)$,  let $(\cpx{Q},\cpx{\partial})$ be a complete   resolution of $U$ with $U=\ker(\partial^0)$. Let $K^i:=\ker(\partial^i)$. Then $K^i\in \gpmodcat{\Sigma}$. We have
\begin{align*}
\Ext_\Sigma^{p}(U,M)
&=\Ext_\Sigma^{j+p}(K^j,M)\\
&=\Hom_{\Db{\modcat{\Sigma}}}(K^j,M[j+p])\\
&\cong \Hom_{\Db{\modcat{\Sigma}}}(K^j,\cpx{L}[j+p]).
\end{align*}
Due to $K^j\in \gpmodcat{\Sigma}$, $\Ext^i_\Sigma(K^j, L^q)=0$ for $i\ge 1$ and $s\le q\le t$. By  \cite[Lemma 1.6]{Kato02}, $\Hom_{\Db{\modcat{\Sigma}}}(K^j,\cpx{L}[j+p])=0$ for $j\ge t-p+1$. Then $\Ext_\Sigma^{p}(U,M)=0$ for each $p\geq 1$, that is, $M\in {\rm Gproj}(\Sigma)^{\bot}$.
\end{proof}

The following result, which is our third criterion,  can be considered as a derived version of Criterion~\ref{Crit: restr-Gproj1}.
\begin{Crit}\label{Crit: restr-Gproj2}
   Assume  that   $\mathbb{R}\Hom_\Gamma(X,\Gamma)$ is quasi-isomorphic to a bounded complex with each term in ${\rm Gproj}(\Lambda)^{\bot}$. Then the functor $X\otimes _\Lambda-:\modcat{\Lambda} \rightarrow \modcat{\Gamma}$   induces  $X\otimes^{\mathbb{L}}_\Lambda-: \fGd{\Lambda}  \ra \fGd{\Gamma}$.
\end{Crit}

\begin{proof}
By \cite[Lemma 2.8]{AKLY17}, the functor $X\otimes _\Lambda-:\modcat{\Lambda} \rightarrow \modcat{\Gamma}$ induces $$X\otimes^{\mathbb{L}}_\Lambda-: \Db{\modcat{\Lambda}}\ra \Db{\modcat{\Gamma}}.$$ It suffices to prove that $X\otimes _\Lambda U \in \fGd{\Gamma}$ for any $U\in {\rm Gproj}(\Lambda)$.

Consider a projective resolution $\cpx{P}$ of $X$ as $\Gamma\otimes_k \Lambda^{\op}$-module:
$$\cdots\ra P^{i}\raf{d^i} P^{i+1}\raf{d^{i+1}} \cdots \ra P^{-1}\raf{d^{-1}} P^0\ra {_\Gamma X_\Lambda}\ra 0$$
with $P^1={_\Gamma X_\Lambda}$.  Denote $$P_{\leq 0}^{\bullet}:=(\cdots\ra P^{i}\raf{d^i} P^{i+1}\raf{d^{i+1}} \cdots \ra P^{-1}\raf{d^{-1}} P^0\ra 0).$$
It follows from $\pd(X_\Lambda)<\infty$ that $\Tor_i^{\Lambda}(X,U)=0$ for any $i>0$ and any $U\in \gpmodcat{\Lambda}$. Indeed, this can
be proved by applying ${}_{\Gamma}X\otimes_{\Lambda}-$ to the complete projective resolution of $U$ and doing dimension shifting.
Thus the complex $\cpx{P}\otimes_\Lambda U$
is exact, and
${_\Gamma X}\otimes_\Lambda U\cong {}_\Gamma X\otimes^{\mathbb{L}}_\Lambda U\cong\cpx{P}_{\leq 0}\otimes_\Lambda U$ in $\Db{\modcat{\Gamma}}$ by Lemma \ref{lem-tensor}.

Let $l:=\mathrm{max}\{\pd(_\Gamma X),\pd(X_\Lambda)\}$ and $N:=\ker(d^{-l+1})$. Then both $_\Gamma N$ and $N_\Lambda$  are projective. Thus ${_\Gamma X}\otimes_\Lambda U$ is  quasi-isomorphic to
$$0\ra N\otimes_\Lambda U\ra P^{-l+1}\otimes_A U\ra \cdots \ra P^{-1}\otimes_\Lambda U\ra P^0\otimes_\Lambda U\ra 0,$$
where $P^{i}\otimes_\Lambda U$ is a projective $\Gamma$-module for $-l+1\le i\le 0$. We will prove that $\Hom_\Gamma(N, \Gamma)\in {\rm Gproj}(\Lambda)^{\bot} $, and then by Criterion~\ref{Crit: restr-Gproj1}, it follows  $N\otimes_\Lambda U\in {\rm Gproj}(\Gamma)$. Consequently, $X\otimes_\Lambda U$ is quasi-isomorphic to a complex in $\fGd{\Gamma}$. This shows that
$X\otimes^{\mathbb{L}}_\Lambda-$ sends a Gorenstein projective module into  $\fGd{\Gamma}$, thereby proving the statement.

It remains to prove $\Hom_\Gamma(N, \Gamma)\in {\rm Gproj}(\Lambda)^{\bot}$. Indeed, denote
$$P_{[-l+1, 0]}^{\bullet}:=(0\ra P^{-l+1}\ra  P^{-l+2}\ra  \cdots \ra P^{-1}\raf{d^{-1}} P^0\ra 0).$$
Applying the functor $\mathbb{R}\Hom_\Gamma(-,\Gamma):\D{\Gamma\otimes_{\bfk}\Lambda^{\rm op}}\ra \D{\Lambda}$ to the triangle in $\D{\Gamma\otimes_{\bfk}\Lambda^{\rm op}}$:
$$N[l-1]\ra \cpx{P}_{[-l+1, 0]}\ra X\ra N[l],$$
we obtain the following triangle in $\D{\Lambda}$:
$$\mathbb{R}\Hom_\Gamma(X,\Gamma)\ra \mathbb{R}\Hom_\Gamma(\cpx{P}_{[-l+1, 0]},\Gamma)\ra \mathbb{R}\Hom_\Gamma(N,\Gamma[-l+1])\ra \mathbb{R}\Hom_\Gamma(X,\Gamma)[1]. $$
Since $P^i$ is a projective $\Gamma\otimes_{\bfk}\Lambda^{\rm op}$-module for each $-l+1\le i\le 0$, $\Hom_\Gamma(P^i,\Gamma)\in \inj (\Lambda)$ and $\mathbb{R}\Hom_\Gamma(\cpx{P}_{[-l+1, 0]},\Gamma)\cong \Hom_\Gamma(\cpx{P}_{[-l+1, 0]},\Gamma)\in \Db{\modcat{\Lambda}} $  has all terms lying  in $\inj (\Lambda)$.
By assumption, $\mathbb{R}\Hom_\Gamma(X,\Gamma)$ has all terms in ${\rm Gproj}(\Lambda)^{\bot}$. Since $\inj (\Lambda) \subseteq {\rm Gproj}(\Lambda)^{\bot}$, it follows  that $\mathbb{R}\Hom_\Gamma(N,\Gamma[-l+1])$ is quasi-isomorphic to a bounded complex
with each term in ${\rm Gproj}(\Lambda)^{\bot}$.  Since $N$ is projective left $\Gamma$-module, $\Hom_\Gamma(N, \Gamma)\cong \mathbb{R}\Hom_\Gamma(N,\Gamma)$ in $\Db{\modcat{\Lambda}}$. By Lemma~\ref{Lem: Gp orthogonal}, $\Hom_\Gamma(N, \Gamma)\in {\rm Gproj}(\Lambda)^{\bot}$.
\end{proof}

\begin{Remark} By the definition of Gorenstein modules, we have $\Lambda\in {\rm Gproj}(\Lambda)^{\bot}$, and therefore ${\rm proj}(\Lambda)$ is contained within ${\rm Gproj}(\Lambda)^{\bot}$. Consequently, if $\mathbb{R}\Hom_\Gamma(X,\Gamma)$ is perfect as complex of $\Lambda$-modules, then   $\mathbb{R}\Hom_\Gamma(X,\Gamma)$ is quasi-isomorphic to a bounded complex with each term in ${\rm Gproj}(\Lambda)^{\bot}$. Thus Criterion~\ref{Crit: restr-Gproj2} generalises Criterion~\ref{Crit: restr-Gproj0}.

However, when $\Lambda$ is Gorenstein, ${\rm Gproj}(\Lambda)^{\bot}$ is exactly the subcategory of  modules of finite projective dimension. In this case, the  thick subcategory generated by ${\rm Gproj}(\Lambda)^{\bot}$ within $ \Db{\modcat{\Lambda}}$ is exactly $\Kb{\pmodcat{\Lambda}}$. Thus, under the Gorenstein condition, the two criteria coincide.
\end{Remark}

Now we apply Criteria~\ref{Crit: restr-Gproj1} and  \ref{Crit: restr-Gproj2} to bounded extensions.
We begin with an observation that follows from \cite[Theorem I (iii) or Proposition 3.7]{OPS19}.
\begin{Prop}
Let $B\subset A$ be a bounded extension. Then ${}_AA\otimes_B-: \stgp{B}\ra  \stgp{A}$ is fully faithful.
\end{Prop}
\begin{Thm}\label{main-thm-def1}
Let $B\subset A$ be a bounded extension.
Suppose that   $\mathbb{R}\Hom_B(A,B)$ is quasi-isomorphic to a bounded complex with each term in ${\rm Gproj}(A)^{\bot}$. Then
$${}_AA\otimes_B-: \stgp{B}\ra  \stgp{A}\mbox{ and } {}_AA\otimes^{\mathbb{L}}_B-: \Gdef(B)\rightleftharpoons  \Gdef(A): \Res^A_B$$
are equivalences.
\end{Thm}

\begin{proof}
Since ${}_AA$ projective as left A-module and $\pd (_BA)$, $\pd(A_B)<\infty$, it follows from Criteria~\ref{Crit: restr-Gproj1} (with $\Lambda=B, \Gamma=A, {}_\Gamma X_\Lambda={}_AA_B$) that the functor ${}_AA\otimes_B-: \modcat{B} \rightarrow \modcat{A}$ induces a functor ${}_AA\otimes_B-: \stgp{B}\ra  \stgp{A}$ and thus ${}_AA\otimes^{\mathbb{L}}_B-:  \fGd{B}/\Kb{\pmodcat{B}}\ra  \fGd{A}/\Kb{\pmodcat{A}} $.
%%be proved by applying ${}_AA\otimes_B-$ to the complete projective resolution of $V$ and doing dimension shifting.
%Then by Lemma \ref{lem-tensor}, ${}_AA\otimes_B^\mathbb{L}V\cong {}_AA\otimes_BV \in  \Db{\modcat{A}}$,
%and then the functor ${}_AA\otimes_B^\mathbb{L}-$ induces a triangle functor from $\fGd{B}$ to $\fGd{A}$. We have thus a commutative diagram
%$$\xymatrix{\stgp{B}\ar[r]^{{}_AA\otimes_B-}\ar[d]^{\simeq}& \stgp{A}\ar[d]^{\simeq}\\
%\fGd{B}/\Kb{\pmodcat{B}}\ar[d]\ar[r]^{{}_AA\otimes_B^\mathbb{L}-} & \ar[d]\fGd{A}/\Kb{\pmodcat{A}}\\
%\Sg(B)\ar[r]^{{}_AA\otimes_B^\mathbb{L}-} &\Sg(A) }$$

Since $A_A$ is projective,  $\pd(_BA)<\infty$ and    $\mathbb{R}\Hom_B(A,B)$ is quasi-isomorphic to a bounded complex with each term in ${\rm Gproj}(A)^{\bot}$, it follows from Criteria~\ref{Crit: restr-Gproj2} (with $\Lambda=A, \Gamma=B, {}_\Gamma X_\Lambda={}_BA_A$) that the functor   $\Res^A_B={}_BA\otimes^{\mathbb{L}}_A-$  restricts to  a triangle functor from $\fGd{A}$ to $ \fGd{B}$. Then we obtain an induced  commutative diagram:
$$\xymatrix{\stgp{B}\ar[r]^{{}_AA\otimes_B- }\ar[d]^{\simeq}& \stgp{A}\ar[d]^{\simeq}\\
\fGd{B}/\Kb{\pmodcat{B}}\ar@{>->}[d]\ar@<1ex>[r]^{ {}_AA\otimes^{\mathbb{L}}_B- } &\ar@<1ex>[l]^{\Res^A_B} \ar@{>->}[d]\fGd{A}/\Kb{\pmodcat{A}}\\
\Sg(B) \ar@<1ex>[r]^{ {}_AA\otimes^{\mathbb{L}}_B- } \ar@{->>}[d]&\Sg(A) \ar@{->>}[d] \ar@<1ex>[l]^{\Res^A_B} \\
 \Gdef(B)\ar@<1ex>[r]^{ {}_AA\otimes^{\mathbb{L}}_B- } &  \ar@<1ex>[l]^{\Res^A_B}  \Gdef(A).}$$
By Theorem~\ref{main-thm-semt}, $${_A}A\otimes^{\mathbb{L}}_{B}-: \Sg(B)\rightleftharpoons\Sg(A): \Res^A_B$$ form an adjoint equivalence which implies that all the horizontal functors are equivalences.
%The equivalences could be also shown by using \cite[Theorem 5.3(2)]{HP17}. In fact, since $A_B$ has finite projective dimension,
%${_A}A\otimes^{\mathbb{L}}_{B}-: \Db{B}\to \Db{\modcat{A}} $ is uniformly bounded in the sense of \cite[Definition 4.1]{HP17}.
\end{proof}

%\begin{Remark} %Assume that
%%   $\mathbb{R}\Hom_B(A,B)$ is perfect as complex of $A$-modules. Since $A$ and hence $A\mbox{-}{\rm proj}$ belong to ${\rm Gproj}(A)^{\bot}$,
%Theorems~\ref{main-thm-semt} and \ref{main-thm-def1} imply that
%$${_A}A\otimes^{\mathbb{L}}_{B}-: \Sg(B)\rightleftharpoons\Sg(A):  \Res^A_B $$ is an adjoint equivalence which induces a singular equivalence of Morita type with level between $A$ and $B$  and that  ${}_AA\otimes_B-: \stgp{B}\ra  \stgp{A} $ is a triangle equivalence.
%Therefore, our results  strengthen   \cite[Theorem I (iii)]{OPS19}.
%
%\end{Remark}

Now we will give the last criteria from a different point of view.
\begin{Crit}\label{Crit: restr-Gproj3}
Assume there is an integer $t$ such that $\Omega ^t({_{\Gamma}X}\otimes_\Lambda U)\in {}^{\perp} {\Gamma}$ for any $U\in {}^{\perp}\Lambda$. Then
$X\otimes^{\mathbb{L}}_\Lambda-$ induces a triangle functor from $\fGd{\Lambda}$ to $\fGd{{\Gamma}}$.
\end{Crit}
\begin{proof} This can be proved similarly as \cite[Lemma 3.8]{QS23}. For the convenience of the reader, we
include here a proof.

As in the proof of Criterion~\ref{Crit: restr-Gproj2}, take a projective resolution $\cpx{P}$ of $X$ as $\Gamma\otimes_k \Lambda^{\op}$-module:
$$\cdots\ra P^{i}\raf{d^i} P^{i+1}\raf{d^{i+1}} \cdots \ra P^{-1}\raf{d^{-1}} P^0\ra {_\Gamma X_\Lambda}\ra 0.$$
Denote $N:=\Omega^l(_\Gamma X_\Lambda)$ for $l>>0$. Then $N$ is projective as left $\Gamma$-module and right $\Lambda$-module.
For any $U\in {\rm Gproj}(\Lambda)$,  $\cpx{P}\otimes_\Lambda U$ is exact, so the sequence
 $$0\ra {}_\Gamma N\otimes_\Lambda U\ra  {}_\Gamma P^{-l+1}\otimes_A U\ra \cdots \ra {}_\Gamma P^{-1}\otimes_\Lambda U\ra {}_\Gamma P^0\otimes_\Lambda U \ra {}_\Gamma X\otimes_\Lambda U\ra 0 $$
 is exact. Hence,  ${}_\Gamma N\otimes_\Lambda U\cong \Omega^l({}_\Gamma X\otimes_\Lambda U)\oplus {}_\Gamma V$ with ${}_\Gamma V\in \pmodcat{{\Gamma}}$, and
$ {}_\Gamma X\otimes^{\mathbb{L}}_\Lambda U  \cong  {_\Gamma X}\otimes_\Lambda U \in \Db{\modcat{\Gamma}}$,  which is   quasi-isomorphic to
$$0\ra {}_\Gamma N\otimes_\Lambda U\ra  {}_\Gamma P^{-l+1}\otimes_A U\ra \cdots \ra {}_\Gamma P^{-1}\otimes_\Lambda U\ra {}_\Gamma P^0\otimes_\Lambda U \ra  0. $$
To show that this complex falls into $\fGd{{\Gamma}}$, since ${}_\Gamma P^j\otimes_\Lambda U$ is projective for $j=-l+1, \dots, 0$, it  remains to prove that  when $l>t$,    ${}_\Gamma N\otimes_\Lambda U$ belongs to $ \gpmodcat{\Gamma}$.
%This can be proved similarly as \cite[Lemma 3.8]{QS23}. For the convenience of the reader, we include here a proof.

First, since ${}^{\perp} {\Gamma}$ is closed under taking syzygies  and  $\Omega ^t({}_\Gamma X\otimes_\Lambda U)\in {}^{\perp} {\Gamma}$, when $l>t$, ${}_\Gamma N\otimes_\Lambda U\cong \Omega^l({}_\Gamma X\otimes_\Lambda U)\oplus {}_\Gamma V$ with ${}_\Gamma V\in \pmodcat{{\Gamma}}$  will  belong  to  ${}^{\perp} {\Gamma}$.

Decompose  a complete $\Lambda$-projective resolution  $(\cpx{Q},\partial^\bullet)$ of $U$    with $U=\ker(\partial^0)$ into short exact sequence
$0\ra U^i\ra Q^{i}\ra U^{i+1}\ra 0$ for all $i\in \mathbb{Z}$ with $U^i:=\ker(\partial^i)$.
Since $N_\Lambda$ is projective, we have exact sequences:
$$0\ra {}_\Gamma N\otimes_\Lambda U^i\ra {}_\Gamma N\otimes_\Lambda Q^{i}\ra {}_\Gamma N\otimes_\Lambda U^{i+1}\ra 0.$$
%By Horseshoe Lemma, these lead to exact sequences
%$$0\ra \Omega^t(N\otimes_\Lambda U^i)\ra \Omega^t( N\otimes_\Lambda Q^{i})\oplus V^i\ra \Omega^t(N\otimes_\Lambda U^{i+1})\ra 0,$$
%with $V^i\in \pmodcat{{\Gamma}}$. Note that $W^i:=\Omega^t( N\otimes_\Lambda Q^{i})\oplus V^i\in \pmodcat{{\Gamma}}$.
Since all $U^i$ are Gorenstein projective, ${}_\Gamma N\otimes_\Lambda U^i \in {}^{\perp} {\Gamma}$. As all ${}_\Gamma N\otimes_\Lambda Q^{i}$ are projective,   we can form a complete resolution of  $N\otimes_\Lambda U$, hence
$N\otimes_\Lambda U \in \gpmodcat{\Gamma}$.

%By assumption, $\Omega^t(X\otimes_\Lambda U^i)\in {}^{\perp}{\Gamma}$ and $U^i$ is Gorenstein projective $\Lambda$-module, then there is an exact sequence
% $$0\ra N\otimes_\Lambda U^i\ra Q^{-l+1}\otimes_\Lambda U^i\ra \cdots \ra Q^{-1}\otimes_\Lambda U^i\ra Q^0\otimes_\Lambda U^i\ra X\otimes_\Lambda U^i \ra0.$$
% Thus we have $\Omega^t(N\otimes_\Lambda U^i)\in {}^{\perp}{\Gamma}$.
% And since $\Omega^t(N\otimes_\Lambda Q^{i})\in \pmodcat{{\Gamma}}$ for all $i\in\mathbb{Z}$. Then $\Omega^t(N\otimes_\Lambda U)\in {\Gamma}\gpmodcat$ and $N\otimes_\Lambda U\in \fGd{\Lambda}$. Thus $X\otimes_\Lambda U\in \fGd{\Lambda}$.
This completes the proof.
\end{proof}

%Now we will prove the second part of Theorem~\ref{main-thm} presented in the
%introduction.
%\begin{Thm}\label{main-thm-def}
%Let $B\subset A$ be a bounded extension.
%Suppose that either $\mathbb{R}\Hom_B(A,B)$ is quasi-isomorphic to a bounded complex with each term in ${\rm Gproj}(A)^{\bot}$
%or the extension $B\subset A$ is split. Then
%$${}_AA\otimes_B-: \stgp{B}\ra  \stgp{A}\mbox{ and }  _AA\otimes^{\mathbb{L}}_B-: \Gdef(B)\rightleftharpoons  \Gdef(A): \Res^A_B$$
%are equivalences.
%\end{Thm}
\begin{Thm}\label{main-thm-def2}
Let $B\subset A$ be a bounded extension.
Suppose that  $_BB$ is a direct summand of $_BA$. Then
$${}_AA\otimes_B-: \stgp{B}\ra  \stgp{A}\mbox{ and }  _AA\otimes^{\mathbb{L}}_B-: \Gdef(B)\rightleftharpoons  \Gdef(A): \Res^A_B$$
are equivalences. In particular, these two equivalences exist if the extension $B\stackrel{i}{\hookrightarrow} A$  splits (i.e. there is an algebra homomorphism $p:A\to B$ such that $p\circ i=\Id_B$) or the left $B$-module $_BA/B$ is projective.
\end{Thm}
\begin{proof}

  Note that $\Res^A_B={_BA}\otimes_A-$ is exact and
$\pd(_BA)=\pd(_B\Res^A_B(A))< \infty$.
Now we claim that there is an integer $t$ such that $\Omega ^t(\Res^A_B(U))\in {}^{\perp} B$ for any $U\in {}^{\perp}A$,
and then Criterion~\ref{Crit: restr-Gproj3} can be applied. The equivalences can be shown as in the last paragraph of  the proof of Theorem~\ref{main-thm-def1}.

Let us show the claim that  there is an integer $t$ such that $\Omega ^t(\Res^A_B(U))\in {}^{\perp} B$ for any $U\in {}^{\perp}A$.

%For the sake of simplicity, we denote by $G=\Res^A_B:B\modcat \rightarrow \modcat{A}$, which extends to a
%triangle functor $G:\D{B\Uodcat}\rightarrow \D{A}$. Denote by $F={}_AA\otimes_B^{\mathbb{L}}-:\D{A}\rightarrow \D{B}$.
%Then $(F,G)$ is an adjoint pair between the corresponding derived categories.
Since  $_BB$ is a direct summand of $_BA$, then
for any $U\in {}^{\perp}A$ and $i>0$, $\Ext _B^i(  \Res^A_B(U), B)$
is a direct summand of $\Ext _B^i(  \Res^A_B(U), _BA)$. On the other hand,
it follows from Theorem~\ref{main-thm-semt2} that the functor $\Res^A_B$ is a $t$-eventually homological isomorphism
for some integer $t$. Therefore, for any $i>t$,
we have isomorphisms $$\Ext _B^i(  \Res^A_B(U), _BA)=\Ext _B^i(  \Res^A_B(U), \Res^A_B(A))
\cong \Ext _A^i(U, A),$$ which is equal to zero since $U\in {}^{\perp}A$.
Hence, for any $i>t$ we have $\Ext _B^i(  \Res^A_B(U), B)=0$,  and for any $j>0$,
 $$\Ext _B^{j}(  \Omega^t(\Res^A_B(U)), B)=\Ext _B^{j+t}(  \Res^A_B(U), B)=0.$$
This shows that $\Omega^t(\Res^A_B(U))\in {}^{\perp} B$.
\end{proof}

\section{Homological conjectures for bounded extensions}\label{Sect:Conjecture}

In this section, we use the results of Sections~\ref{semt} and \ref{Sect: Gorenstein}  to deal with several homological conjectures for bounded extensions, which are still open. They are stated as follows.

{\bf The finitistic dimension conjecture \cite{Bass60}}: \emph{The finitistic dimension of $A$ which is the supremum of the projective dimensions of finitely generated $A$-modules with finite projective dimension, is finite.}

{\bf Han's conjecture \cite{Han06}}: \emph{If the Hochschild homology groups $\HH_n(A)$ vanish for all sufficiently large $n$, then its global dimension is finite.}

{\bf Keller's conjecture \cite{Keller}}: \emph{Let ${\bf S}_{\dg}(A)$ be the dg enhancement of the singularity category $\Sg(A)$. Then there is an isomorphism in the homotopy category of $B_\infty$-algebras
$$\C^*_{\sg}(A, A)\cong \C^*({\bf S}_{\dg}(A), {\bf S}_{\dg}(A)),$$
where $\C^*_{\sg}(A,A)$ is the singular Hochschild cohomology complex \cite{Wang21} of $A$, and $\C^*({\bf S}_{\dg}(A), {\bf S}_{\dg}(A))$ is the Hochschild cochain complex of ${\bf S}_{\dg}(A)$.
}

{\bf Auslander and Reiten's Conjecture \cite{AR75}}: \emph{A finitely generated $A$-module $X$ satisfying $\Ext_A^i(X, X\oplus A)=0$ for all $i\ge 1$ must be projective.}

From \cite{SS04, EHSST04}, an algebra $A$ is said to satisfy {\bf the Fg condition} if the
Hochschild cohomology ring $HH^*(A)$ is Noetherian
and the Yoneda algebra $\Ext _A^*(A/ \rad A, A/ \rad A)$ is a finitely generated
$HH^*(A)$-module. The Fg condition enables developing support variety theory via Hochschild cohomology \cite{SS04, EHSST04}.

Note that for  Han's conjecture  the following result  was firstly proved by  Cibils, Lanzilotta, Marcos, and Solotar for left or right bounded extensions  in \cite{CLMS22} generalising their previous results in \cite{CLMS20b}.
\begin{Thm}\label{main-cor-semt-1}
Let $B\subset A$ be a bounded extension. Then the following statements hold:
\begin{itemize}
\item[$(1)$] $B$ satisfies the finitistic dimension conjecture if and only if so does $A$.

\item[$(2)$]  $B$ satisfies Han's conjecture if and only if so does $A$.

\item[$(3)$] $B$ satisfies Keller's conjecture if and only if so does $A$.
\end{itemize}
\end{Thm}
\begin{proof} By Theorem \ref{main-thm-semt}, we know that $A$ and $B$ are singularly equivalent of Morita type with level,  which induces an equivalence of singularity categories, thus  the finiteness of the global dimension is preserved.
Note that this equivalence
preserves  the finitistic dimension conjecture \cite{Wang15},  Hochschild homology \cite{Wang15} and Keller's conjecture \cite{CLW20}. Thus the statements
(1)-(3) follow.
\end{proof}

\begin{Remark}
%The assumption in Theorem \ref{main-cor-semt-1} regarding the bounded extension is less restrictive than that imposed on the left or %right bounded extension as introduced in \cite{CLMS22}. Consequently,
\begin{itemize}
\item[$(1)$]
Theorem \ref{main-cor-semt-1}(2) extends the main result presented in \cite{CLMS22} from left or right bounded extensions to bounded extensions. Additionally, the approach employed to establish Theorem \ref{main-cor-semt-1}(2) differs from that adopted in \cite{CLMS22}, where the Jacobi-Zariski long nearly exact sequence is utilized \cite{CLMS21b}.
\item[$(2)$] More results  about  Han's conjecture for Morita context algebras are obtained in \cite{CRS21} and \cite{CLMS24}.
In a forthcoming paper we will study Han's conjecture from the viewpoint of recollements and ladders of derived categories  with applications to  Morita context algebras in mind.
\item[$(3)$] In 2021, Iusenko and MacQuarrie introduced a different generalization of left or right bounded extensions, termed the strongly proj-bounded extensions (see \cite{IM21}). They showed that, for such extensions, the finiteness of the global dimension, the finiteness of the big finitistic dimension, and the support of the Hochschild homology are preserved, thus generalizing the results of Cibils, Lanzillota, Marcos, and Solotar in \cite{CLMS22}. Furthermore, MacQuarrie and Naves showed that if an extension satisfies only Condition (2) of the bounded extension as defined in Definition \ref{def-bound-exten}, then the finiteness of the finitistic dimension of the larger algebra implies that of the smaller algebra (see \cite{MN23}).
 \end{itemize}
\end{Remark}

\begin{Thm}\label{main-cor-semt-2}
Let $B\subset A$ be a bounded extension. Assume that ${}_AA^{\tr_{B}}=\mathbb{R}\Hom_{B}({}_BA_A,{}_BB)$ is perfect as left $A$-module
or that $_BB$ is a direct summand of $_BA$.
Then the following statements hold:
\begin{itemize}
\item[$(1)$] $B$ satisfies Auslander-Reiten's conjecture if and only if so does $A$.

\item[$(2)$] $B$ satisfies the Fg condition
if and only if so does $A$. %In this case, they have the same support variety theory.

\end{itemize}
\end{Thm}

\begin{proof}

(1) By the exact sequence $0\rightarrow B \rightarrow A \rightarrow M \rightarrow 0$, since the projective dimension of $M$ as $B$-$B$-bimodule is finite, $\pd(_BA)$ and $\pd(A_B)$ are finite.
%Then
%$\mathbb{R}\Hom_{B}(\mathbb{R}\Hom_{B^{\op}}(A,B),B)\cong {_AA_B}$
%$\mathbb{R}\Hom_{B}(A,-)\cong\mathbb{R}\Hom_{B}(A,B)\otimes_B^{\mathbb{L}}-$
%
Then we have the following  adjoint 5-tuple of  triangle functors
\begin{align*}
\xymatrixcolsep{5pc}\xymatrix{\mathrm{D}(B)
\ar@<1.5pc>[r]|{{}_AA\otimes^{\mathbb{L}}_{B}-}
\ar@<-1.5pc>[r]|{\mathbb{R}\Hom_{B}(_BA_A,-)}
&\mathrm{D}(A)  \ar[l]|{\Res^A_B={}_BA\otimes_{A}-}
\ar@<3pc>[l]^{ \mathbb{R}\Hom_{A}(_A(A^{\tr_{B}})_B,-)}
\ar@<-3pc>[l]_{_BA^{\tr_{B^\op}}\otimes^{\mathbb{L}}_{A}-}
}
\end{align*}
where $A^{\tr_{B^\op}}=\mathbb{R}\Hom_{B^\op}(_AA_B, B)$,
and $(_AA\otimes ^L_B-,_BA\otimes _A-)$ induces mutually inverse
singular equivalences between $A$ and $B$. %where $\Modcat{A}$ stands for the category of all left $A$-modules and $\mathrm{D}(A)$ stands for the unbounded derived categories of complexes over $\Modcat{A}$.
Since $_BA\otimes _A-$ preserves $\Kb{\proj}$ and coproducts,
we get that $_BA\otimes _A-$ preserves $\Kb{\Proj}$.
Then it follows from \cite[Corollary 3.5]{CHQW23} that
Auslander-Reiten's conjecture holds for $B$ if it holds for $A$.
Now we will prove the converse statement
under the following two cases.

{\it Case 1.} Assume $A^{\tr_{B}}=\mathbb{R}\Hom_{B}({}_BA_A,{}_BB)$ is perfect as left $A$-module. Then the functor $\mathbb{R}\Hom_{B}({}_BA_A,-)\cong {}_AA^{\tr_{B}}\otimes_B^{\mathbb{L}}-$ preserves perfect complexes and the diagram extends downwards one step further. Thus we complete the proof by \cite[Corollary 3.7]{CHQW23}.

{\it Case 2.} Assume the extension $B\subset A$   splits. Then
by the proof of Theorem~\ref{main-thm-def2},
there is an integer $t$ such that $\Omega ^t(\Res^A_B(U))\in {}^{\perp} B$ for any $U\in {}^{\perp}A$.
Therefore, there is a commutative diagram
$$\xymatrix@!=1pc{ \underline{^\bot A} \ar @{^{(}->}[d]
\ar[rr]^{\Omega ^t \circ \Res^A_B} &&\underline{^\bot B} \ar @{^{(}->}[d] \\
\Sg(A)
\ar[rr]^{[-t]\circ \Res^A_B} &&\Sg(B)
}$$ where $[1]$ is the shift functor on singularity category,
$\underline{{}^{\perp}A}$ is the additive quotient category of ${}^{\perp}A$ modulo projective modules
and the vertical maps are natural full embedding, see \cite[Lemma 2.1]{CHQW23} for example.
Since $[-t]\circ \Res^A_B: \Sg(A)\rightarrow \Sg(B)$ is an equivalence,  it follows from \cite[Lemma 3.3]{CHQW23}
that Auslander-Reiten's conjecture holds for $A$ if it holds for $B$.

(2) Since ${\rm proj}(A)$  is contained in ${\rm Gproj}(A)^{\bot}$, it follows from
Theorem~\ref{main-thm-def1} and Theorem~\ref{main-thm-def2} that $\Gdef(B) \simeq  \Gdef(A)$.
Therefore, $B$ is Gorenstein if and only if so does $A$.
 Assume that $B$ satisfies the Fg condition. Then it follows from \cite[Theorem 1.5 (a)]{EHSST04}
that $B$ is a Gorenstein algebra, and so does $A$. On the other hand,
it follows from Theorem~\ref{main-thm-semt} that there is a singular equivalence of Morita type with level between
$A$ and $B$. Therefore, $A$ satisfies the Fg condition by \cite[Theorem 7.4]{Ska16},
and similarly, if $A$ satisfies the Fg condition then so does $B$.
%In this case, they have the same support variety theory by \cite{Chen18}.
\end{proof}

\section{Examples}\label{Sect:Examples}
In this section, we apply Theorems~\ref{main-thm-semt}, \ref{main-thm-semt2}, \ref{main-thm-def1},  \ref{main-thm-def2}, \ref{main-cor-semt-1} and \ref{main-cor-semt-2}  to many examples, and we get several reduction methods on singularity categories and Gorenstein defect categories.

Let $B\subset A$ be an algebra extension which splits. We will describe the construction of split algebra extensions. Let $i: B\to A$ be the algebra inclusion and $q: A\to B$   be the retraction. Denote $M=\Ker(q)$ and $j: M\to A$ be the inclusion.  Then there is a linear map $p: A\to M$ such that $j\circ p+i\circ q=\Id_A$, that is, we have the splitting:
$$\xymatrix{0 \ar[r]  &  B\ar@<1ex>[r]^i & A \ar@<1ex>[l]^q \ar@<1ex>[r]^p &  M\ar@<1ex>[l]^j \ar[r]&  0}$$
and we have the vector space isomorphism $A\cong B\oplus M$.
Since $q: A\to B$ is an algebra homomorphism, $M$ is a $B$-$B$-bimodule. The multiplication on $A$ induces, via the vector space isomorphism $A\cong B\oplus M$,  a multiplication on $M$, $\cdot: M\otimes_B M\to M$ which is also a $B$-$B$-bimodule homomorphism.
In this case, the multiplication on $B\oplus M$ will be given by
$$(b, x)(b', x')=(bb', bx'+xb'+x\cdot x').$$
This motivated the following definition as well as a complete description of split algebra extensions.
\begin{Def}[{\cite[Definition 2.3]{BGN12}}] Let $B$ be algebra and $M$ a   $B$-$B$-bimodule. The bimodule ${}_BM_B$ is called a bimodule algebra if there is a $B$-$B$-bimodule homomorphism
$$\cdot: M\otimes_B M\to M$$
subject to the following properties:
  $$(x\cdot y)\cdot z= x\cdot (y\cdot z), x, y, z\in M.$$

\end{Def}

\begin{Prop}[{\cite[Proposition 2.4]{BGN12}}] Let $B$ be algebra and $M$ a   $B$-$B$-bimodule algebra with product $\cdot: M\otimes_B M\to M$. Then the vector space
$A=B\oplus M$ has an algebra structure given by
$$(b, x)(b', x')=(bb', bx'+xb'+x\cdot x'), b, b'\in B, x, x'\in M.$$
Moreover, the inclusion $B\to A, b\mapsto (b, 0)$ is an algebra extension which splits via $q: A\to B, (b, x)\mapsto b$.

Conversely, each split algebra extension $B\subset A$ is given as above.
\end{Prop}

%We can restate Theorems~\ref{main-thm-semt} and \ref{main-thm-def2} for split algebra extensions.
\begin{Prop}\label{Main-for-split-extension} Let $B\subset A=B\oplus M$ be a  split algebra extension. Then it is a bounded extension if and only if $\pd( _{B}M_B)<\infty$, $M^{\otimes _B p}=0$ for some integer $p$ and
$\Tor _i^B(M, M^{\otimes _B j})=0$ for each $i,j\geq 1$.
Hence, Theorems~\ref{main-thm-semt}, \ref{main-thm-semt2},    \ref{main-thm-def2}, \ref{main-cor-semt-1} and \ref{main-cor-semt-2}  apply.
%In this case, there are triangle equivalences:
%$$ \Sg(B) \simeq \Sg(A),\    \stgp{B}\simeq \stgp{A},\   \mbox{and} \  \Gdef(B)\simeq \Gdef(A).$$
%Moreover, $A$ and $B$ are  singularly equivalent of Morita type with level.
\end{Prop}

Let us specify Proposition~\ref{Main-for-split-extension} for trivial extensions and further to triangular matrix algebras.

 Let $B$ be a finite dimensional  algebra and $M$   a finite dimensional $B$-$B$-bimodule.
Let $A=B\ltimes M$ be the trivial extension of $A$ by $M$.
Obviously, the algebra inclusion $B\hookrightarrow A$ defined by $b\mapsto (b,0)$
has a retraction $A\rightarrow B$ given by $(b,m)\mapsto b$. Therefore, the extension
$B\subset  A$ splits with ${}_B(A/B)_B={}_BM_B$.
Trivial extensions of algebras are exactly split algebra extensions with trivial product on the bimodule algebras.
Following Proposition~\ref{Prop-bounded-exten}, we have the following result.
\begin{Prop}\label{prop-trivial-extension}
 The trivial extension $B\subset A=B\ltimes M$ is a bounded extension if and only if $\pd( _{B}M_B)<\infty$, $M^{\otimes _B p}=0$ for some integer $p$ and
$\Tor _i^B(M, M^{\otimes _B j})=0$ for each $i,j\geq 1$.
Hence, Theorems~\ref{main-thm-semt}, \ref{main-thm-semt2},    \ref{main-thm-def2}, \ref{main-cor-semt-1} and \ref{main-cor-semt-2}  apply.

%If this is the case,  there are triangle equivalences
%$$ \Sg(B) \simeq \Sg(A), \  \stgp{B}\simeq \stgp{A},\   \mbox{and} \   \Gdef(B)\simeq \Gdef(A).$$
%Moreover, $A$ and $B$ are singularly equivalent of Morita type with level.
\end{Prop}

\begin{Exam}\label{exam-1}
Let $A$ be an admissible quotient $\bfk Q/I$ of a path algebra $\bfk Q$ over a field $\bfk$.
Consider an arrow $\alpha$ in $Q$ such that $\alpha$ does not occur in a minimal generating set
of $I$. The {\it arrow removal algebra} $B =A/\langle \overline{\alpha} \rangle$ of $A$
was investigated in \cite{CLMS20, DS08, EPS22, GPS21}. Indeed, it follows from \cite[Theorem A]{GPS21} that $A=B\ltimes M$ where
$M$ is projective as a $B$-$B$-bimodule and $M\otimes _BM=0$.
Therefore, $B\subset A$ is a bounded split extension.
Applying Theorem~\ref{main-thm-semt}, \ref{main-thm-semt2}, \ref{main-thm-def2},
\ref{main-cor-semt-1}(1) and \ref{main-cor-semt-2}(2),
we reobtain \cite[Theorem A (ii)]{GPS21} and the main result of \cite{EPS22}.
We mention that the Auslander-Reiten conjecture, Gorenstein defect categories and stable categories of
Gorenstein projective modules between $A$ and $B$
were investigated in \cite[Corollary I]{QS23}, which can also be deduced by
Theorem~\ref{main-thm-def2} and Theorem~\ref{main-cor-semt-2}(1).
\end{Exam}

Let $\Lambda$ and $\Gamma$ be two finite dimensional algebras, $_{\Gamma}W_{\Lambda}$ be a $\Gamma$-$\Lambda$-bimodule
and
$A=\left(\begin{smallmatrix}
\Lambda & 0 \\
{_{\Gamma}W_{\Lambda}} & \Gamma                                           \end{smallmatrix}\right)$
be a triangular matrix algebra. In \cite{Lu17}, Lu proved that
if $W$ is projective as a $\Gamma$-$\Lambda$-bimodule, then there are triangle equivalences
$$\Sg(A) \simeq \Sg(\Lambda)\coprod  \Sg(\Gamma)\  \mathrm{and}\ \Gdef(A)\simeq \Gdef(\Lambda)\coprod \Gdef(\Gamma) $$
(see the proof of \cite[Proposition 4.2 and Theorem 4.4]{Lu17}). Now we will generalize
these results as follows.

\begin{Prop}\label{prop-tri-matrix-alg}
Keep the above notations.
If $\pd({}_\Gamma W_ \Lambda)<\infty$,
then there are triangle equivalences:
$$\Sg(A) \simeq \Sg(\Lambda)\coprod  \Sg(\Gamma),  \stgp{A}\simeq \stgp{\Lambda}\coprod \stgp{\Gamma},\  \mbox{and} \ \Gdef(A)\simeq \Gdef(\Lambda)\coprod \Gdef(\Gamma).$$
\end{Prop}
\begin{proof}
It is clear that $$A=\left(\begin{smallmatrix}
\Lambda & 0 \\
{_{\Gamma}W_{\Lambda}} & \Gamma                                           \end{smallmatrix}\right)
\cong (\Lambda \times \Gamma)\ltimes W,  W\mathop{\otimes}\limits_{\Gamma\times \Lambda}W=0, \ \mathrm{and}\
 \Tor _i^{\Lambda\times \Gamma}(W, W)=0,\  \mathrm{for}\ i\geq 1.$$ Therefore,
the statement follows from Proposition~\ref{prop-trivial-extension}.
\end{proof}

Now we will illustrate our results by two  examples.
\begin{Exam}\label{exam-2}
{\rm Let $\Lambda = \bfk Q/I$ be the algebra where $Q$ is the quiver
$$\xymatrix{1 \ar@<+1ex>[r]^\beta  \ar@(ul,dl)_\gamma & 2 \ar@<+1ex>[l]^\alpha }$$
and $I=\langle \gamma ^2, \alpha \beta\rangle$. We write the concatenation of paths from the right to the left.
Let $\Gamma=\Lambda/\Lambda \alpha \Lambda$. Then $\Gamma$ is the quiver
$$\xymatrix{1 \ar[r]^\beta  \ar@(ul,dl)_\gamma & 2  }$$
with relation $\{ \gamma ^2\}$, and $\Gamma$ is a subalgebra of
$\Lambda$. Therefore, $\Lambda \alpha \Lambda$ can be viewed as a $\Gamma$-$\Gamma$-bimodule,
and there is an algebra isomorphism $\Lambda\cong\Gamma \ltimes \Lambda \alpha \Lambda$
mapping $\overline{\varepsilon}_i$ to $(\overline{\varepsilon}_i,0)$, $\overline{\gamma}$ to $(\overline{\gamma},0)$,
$\overline{\beta}$ to $(\overline{\beta},0)$ and $\overline{\alpha}$ to $(0,\overline{\alpha})$,
where $\varepsilon _i$ is the trivial path at $i$. It is easy to check that $\Lambda \alpha \Lambda _\Gamma
\cong (e_2\Gamma/\rad(e_2\Gamma))^4$ and $_\Gamma \Lambda \alpha \Lambda
\cong \Gamma e_1$. Therefore, we get $\Lambda \alpha \Lambda \otimes _\Gamma \Lambda \alpha \Lambda =0$
and $\Tor _i^\Gamma(\Lambda \alpha \Lambda, \Lambda \alpha \Lambda)=0$ for each $i\geq 1$.
Now we claim that $\pd _{\Gamma^e}(\Lambda \alpha \Lambda)<\infty$, and then
we get triangle equivalences
$\Sg(\Lambda) \simeq \Sg(\Gamma)$ and $\Gdef(\Lambda)\simeq \Gdef(\Gamma)$ by Proposition~\ref{prop-trivial-extension}.
Indeed, the enveloping algebra $\Gamma^e$ has the following quiver
\[\scriptsize \xymatrix@C=1.5cm@R=1.5cm{
  1\times 1^{\mathrm{op}}
  \ar@(ul,dl)_{\gamma\otimes 1^{\mathrm{op}}}
  \ar@(ul,ur)^{1\otimes \gamma^{\mathrm{op}}}
  \ar[r]^{\beta\otimes 1^{\mathrm{op}}}
  \ar@{<-}[d]_{1\otimes\beta^{\mathrm{op}}}
& 2\times 1^{\mathrm{op}}
  \ar@(ul,ur)^{2\otimes\gamma^{\mathrm{op}}}
  \ar@{<-}[d]_{2\otimes\beta^{\mathrm{op}}}
\\1\times 2^{\mathrm{op}}
  \ar[r]_{\beta\otimes 2^{\mathrm{op}}}
  \ar@(ul,dl)_{\gamma\otimes 2^{\mathrm{op}}}
& 2\times 2^{\mathrm{op}} }\]
with relations \small{$\{ (\gamma\otimes 1^{\mathrm{op}})^2, (1\otimes \gamma^{\mathrm{op}})^2,
(2\otimes\gamma^{\mathrm{op}})^2, (\gamma\otimes 2^{\mathrm{op}})^2,
(2\otimes\beta^{\mathrm{op}})(\beta\otimes 2^{\mathrm{op}})-(\beta\otimes 1^{\mathrm{op}})(1\otimes\beta^{\mathrm{op}}),
(\gamma\otimes 1^{\mathrm{op}})(1\otimes\beta^{\mathrm{op}})-(1\otimes\beta^{\mathrm{op}})(\gamma\otimes 2^{\mathrm{op}}),
(2\otimes\gamma^{\mathrm{op}})(\beta\otimes 1^{\mathrm{op}})-(\beta\otimes 1^{\mathrm{op}})(1\otimes \gamma^{\mathrm{op}}),
(\gamma\otimes 1^{\mathrm{op}})(1\otimes \gamma^{\mathrm{op}})-(1\otimes \gamma^{\mathrm{op}})(\gamma\otimes 1^{\mathrm{op}})\}$.}
The $\Gamma^e$-module $\Lambda \alpha \Lambda$ is given by the following representation:

\medskip

\[\scriptsize \xymatrix@C=1.5cm@R=1.5cm{
  0
  \ar@(ul,dl)
  \ar@(ul,ur)
  \ar[r]
  \ar@{<-}[d]
& 0
  \ar@(ul,ur)
  \ar@{<-}[d]
\\ \bfk^2
  \ar[r]_{\left(\begin{array}{cc} 1 & 0 \\ 0 & 1  \end{array}\right)}
  \ar@(ul,dl)_{ \left(\begin{array}{cc} 0 & 0 \\ 1 & 0  \end{array}\right)}
& \bfk^2 \ .}\]
There is an exact sequence
$$0\rightarrow \Gamma^ee_{1\times 1^{\mathrm{op}}}\rightarrow \Gamma^ee_{1\times 2^{\mathrm{op}}}
\rightarrow \Lambda \alpha \Lambda \rightarrow 0$$
of $\Gamma^e$-modules, and thus $\pd (_{\Gamma^e}\Lambda \alpha \Lambda)=1$.

}
\end{Exam}

\begin{Exam} This (non-)example is taken from \cite[Example 5.5]{PSS14}.

Consider the triangular matrix algebra
$$A=\left(\begin{matrix}
\bfk & 0 \\
\bfk& \bfk[x]/(x^2)                                          \end{matrix}\right).$$
Since the bimodule $W=\bfk$ is not a $\bfk[x]/(x^2)$-$\bfk$-bimodule of finite projective dimension, the algebra extension
$B=\bfk\times \bfk[x]/(x^2)\subset A$ is NOT a bounded extension.  In fact,     $B$ is a Gorenstein algebra, while $A$ is NOT, so their Gorenstein defect categories are not equivalent, although the extension $B\subset A$ splits.

\end{Exam}
\noindent
{\bf Acknowledgments.} The first author was supported by the National Natural Science Foundation of China (No. 12061060, 11961007),
by the project of Young and Middle-aged Academic and Technological
leader of Yunnan (Grant No. 202305AC160005) and by Yunnan Key Laboratory of Modern Analytical Mathematics and Applications (No. 202302AN360007).
The second and the fourth authors were supported by the National Natural Science Foundation of China (No. 12071137), by Key Laboratory of Mathematics and Engineering Applications (Ministry of Education), by Shanghai Key Laboratory of PMMP (No. 22DZ2229014) and by Fundamental Research Funds for the Central Universities.

 We thank Hongxing Chen,  Xiao-Wu Chen, Claude Cibils,  Changjian Fu, Henning Krause, Andrea Solotar,  Ren Wang   for helpful remarks.

\noindent
{\bf Declaration of interests.} The authors have no conflicts of interest to disclose.

\noindent
{\bf Data availability.} No new data were created or analyzed in this study.

\vspc

\end{document}